\documentclass[final,leqno]{siamltex}

\usepackage{amsfonts}
\usepackage{amsmath}
\usepackage{bbm}
\usepackage{graphicx}
\DeclareMathOperator*{\esssup}{ess\,sup}

\DeclareMathOperator*{\essinf}{ess\,inf}

\newtheorem{assumption}[theorem]{Assumption}

\title{Dynkin games with Poisson random intervention times\thanks{The work
is partially supported by a start-up research fund from the
University of Warwick and NSFC No. 11771158.}}

\author{Gechun Liang,\ \ Haodong Sun\thanks{Department of Statistics, University of Warwick, Coventry, CV4 7AL, U.K.
Email adress: {\tt g.liang@warwick.ac.uk;}\ \ {\tt
h.sun.9@warwick.ac.uk}}}

\begin{document}

\maketitle

\begin{abstract}
This paper introduces a new class of Dynkin games, where the two
players are allowed to make their stopping decisions at a sequence
of exogenous Poisson arrival times. The value function and the
associated optimal stopping strategy are characterized by the
solution of a backward stochastic differential equation. The paper
further provides a replication strategy for the game, and applies
the model to study the optimal conversion and calling strategies of
convertible bonds, and their asymptotics when the Poisson intensity
goes to infinity.
\end{abstract}

\begin{keywords}
constrained Dynkin game, \and penalized BSDE, \and optimal stopping
strategy,  \and replication strategy, \and convertible bond.
\end{keywords}

\begin{AMS} 60G40, \and 91A05, \and 91G80, \and 93E20.
\end{AMS}

\pagestyle{myheadings} \thispagestyle{plain} \markboth{Gechun Liang
and Haodong Sun}{Dynkin games with Poisson random intervention
times}


\section{Introduction}
Dynkin games are the games on stopping times, where two players
determine their optimal stopping times as their strategies. The game
was first introduced by Dynkin \cite{dynkin1969game}, and later
generalized by Neveu \cite{neveu1975discrete} in 1970s. In this
game, two players observe two stochastic processes, say $L$ and $U$,
and their aims are to maximize/minimize the expected value of the
payoff
$$R(\sigma,\tau)=L_{\tau}\mathbbm{1}_{\{\tau\leq \sigma\}}+U_{\sigma}\mathbbm{1}_{\{\sigma< \tau\}}$$
over stopping times $\tau$ and $\sigma$, respectively. In a
discrete-time setting, under the assumption that $U\geq L$, Neveu
proved the existence of the game value and its associated optimal
strategy.

Since then, there has been a considerable development of Dynkin
games. The corresponding continuous time models were developed,
among others, by Bismut \cite{bismut1977probleme}, Alario-Nazaret et
al \cite{alario1982dynkin}, Lepeltier and Maingueneau
\cite{lepeltier1984jeu} and Morimoto \cite{morimoto1984dynkin}. In
order to relax the condition $U\geq L$, Yasuda
\cite{yasuda1985randomized} proposed to extend the class of
strategies to randomized stopping times, and proved that the game
value exists under merely an integrability condition. Rosemberg et
al \cite{rosenberg2001stopping}, Touzi and Vielle
\cite{touzi2002continuous} and Laraki and Solan
\cite{laraki2005value} further extended his work in this direction.
If the two players in the game are with asymmetric payoffs, then it
gives arise to a nonzero-sum Dynkin game. See, for example, Hamadene
and Zhang \cite{hamadene2010} and more recently De Angelis et al
\cite{De Angelis2015} with more references therein. A robust version
of Dynkin games can be found in Bayraktar and Yao
\cite{Bayraktar2017} if the players are ambiguous about their
probability model.

The setups in all the aforementioned works are either in continuous
time where stopping times take any value in a certain time interval,
or in discrete time where stopping times only take values in a
pre-specified time grid. In this paper, we consider a hybrid of
continuous and discrete times, and introduce a new type of Dynkin
games, where both players are allowed to stop at a sequence of
random times generated by an exogenous Poisson process serving as a
signal process. We call such a Dynkin game a \emph{constrained
Dynkin game}.

The underlying Poisson process can be regarded as an exogenous
constraint on the players' abilities to stop, so it may represent
the liquidity effect, i.e. the Poisson process indicates the times
at which the underlying stochastic processes are available to stop.
Moreover, the Poisson process can also be seen as an information
constraint. The players are allowed to make their stopping decisions
at all times, but they are only able to observe the underlying
stochastic processes at Poisson times.

Our first main result is Theorem \ref{bigTheorem}, which
characterizes the value of the constrained Dynkin game and its
associated optimal stopping strategy in terms of the solution of a
penalized backward stochastic differential equation (BSDE). The
latter is widely used to approximate the solution of a reflected
BSDE with double obstacles and the corresponding continuous time
Dynkin game. The main idea to solve the constrained Dynkin game is
to introduce a family of auxiliary games (see
(\ref{aux_upperValues2})-(\ref{aux_lowerValues2})), for which
standard dynamic programming principle holds. Furthermore, following
from the convergence of penalized BSDE to reflected BSDE (see, for
example, \cite{cvitanic1996backward} and
\cite{hamadene1999infinite}) and the penalized BSDE characterization
(\ref{final_game_value}) of the constrained Dynkin game, we also
make a connection with standard Dynkin games in continuous time.
That is, the value of the constrained Dynkin game will converge to
the value of its continuous time counterpart when the Poisson
intensity goes to infinity.

Our second main result is about replication of the constrained
Dynkin game (see Theorem \ref{theorem_5.1}). This has an application
to the hedging problems in finance. In the existing literature of
financial applications of optimal stopping with Poisson times, the
vast majority of papers focus on the risk-neutral valuation without
even mentioning the issue of hedging (see \cite{dupuis2002optimal}
and \cite{lempa2012optimal} among others). This somewhat lacks a
foundation since, as is well known, the major argument supporting
the risk-neutral valuation is the existence of hedging strategies.
We address this issue by constructing a replication strategy for the
constrained Dynkin game (which in particular covers the optimal
stopping case). For such a replication problem, a new element is the
jump risk stemming from the Poisson process. To hedge this jump
risk, we introduce a pricing process generated by the jump times of
the Poisson process. We then construct the replication strategies
recursively for a sequence of constrained Dynkin games starting from
different Poisson arrival times, and for each game, the replication
strategy is constructed via two linear BSDEs. The first BSDE is used
to replicate the payoff of the game before the next jump time, and
the second equation is used to replicate the payoff after this jump
time.

With the above replication strategies behind the risk-neutral
valuation, we then apply the constrained Dynkin game to study
convertible bonds. In a convertible bond, the bondholder decides
whether to keep the bond to collect coupons or to convert it to the
firm's stocks. She will choose a conversion strategy to maximize the
bond value. On the other hand, the issuing firm has the right to
call the bond, and presumably acts to maximize the equity value of
the firm by minimizing the bond value. This creates a two-person,
zero-sum Dynkin game.

Traditionally, convertible bond models often assume that both the
bond holder and the firm are allowed to stopped at any stopping time
adapted to the firm's fundamental (such as its stock prices). In
reality, there may exist some liquidation constraint as an external
shock, and both players only make their decisions when such a shock
arrives. We model such a liquidation shock as the arrival times of
an exogenous Poisson process, and thus the convertible bond model
falls into the framework of constrained Dynkin games. A similar idea
has first appeared in the modeling of debt run problems (see
\cite{Liang2}), which can be formulated as optimal stopping problems
with Poisson arrival times.

Furthermore, in a Markovian setting, we derive explicitly the
optimal stopping strategies for both the bondholder and the firm. We
show that if the initial stock price is not too high (otherwise the
game will stop at the first Poisson arrival time), the optimal
stopping rules of the two players depend on the relationship between
the coupon rate $c$, dividend rate $q$, interest rate $r$ and
surrender price $K$. For the firm, its optimal stopping strategy is
to either call the bond back as soon as possible (if $c\geq rK$) or
postpone the calling time of the bond as late as possible (if
$c<rK$). In contrast, the investor's optimal stopping strategy
depends on the relationship between $c$ and $qK$. If $c>qK$, the
investor will delay her conversion time as late as possible; if
$c\leq qK$, her conversion strategy is determined by an optimal
conversion boundary, the latter of which is obtained by solving a
free boundary problem.

Turning to the literature, the optimal stopping problem with
constraints on the stopping times was introduced by Dupuis and Wang
\cite{dupuis2002optimal}, when they used it to model perpetual
American options exercised at exogenous Poisson arrival times. See
also Lempa \cite{lempa2012optimal} and Menaldi and Robin
\cite{menaldi2016some} for further extensions of this type of
optimal stopping problems. On the other hand, Liang
\cite{liang2015stochastic} made a connection between such kind of
optimal stopping problems with penalized BSDE. The corresponding
optimal switching (impulse control) problems were studied by Liang
and Wei \cite{liang2013optimal} and more recently by Menaldi and
Robin \cite{menaldi2017some} with  more general signal times and
state spaces.

The study of convertible bonds dated back to Brennan and Schwartz
\cite{brennan1977convertible} and Ingersoll
\cite{ingersoll1977contingent}. However, it was Sirbu et al
\cite{si2004perpetual} who first analyzed the optimal strategy of
perpetual convertible bonds (see also Sirbu and Shreve \cite{si2006}
for the finite horizon counterpart). They reduced the problem from a
Dynkin game to an optimal stopping problem, and discussed when call
precedes conversion and vice versa. Several more realistic features
of convertible bonds have been taken into account since then. For
example, Bielecki et al \cite{bielecki2008arbitrage} considered the
problem of the decomposition of a convertible bond into bond
component and option component. Crepey and Rahal
\cite{crepey2011pricing} studied the convertible bond with call
protection, which is typically path dependent. Chen et al
\cite{chen2013nonzero} considered the tax benefit and bankruptcy
cost for convertible bonds. For a complete literature review, we
refer to the aforementioned papers with references therein.




The paper is organized as follows. Section 2 contains the problem
formulation and main result, with its proof provided in section 3.
In section 4, we establish a connection with standard Dynkin games.
Section 5 is about replication of the constrained Dynkin game. In
section 6, we apply the constrained Dynkin game to study the
convertible bonds in a Markovian setting, and derive the explicit
optimal stopping strategies and the corresponding free boundaries
under various situations. Section 7 carries out an asymptotic
analysis of the game values and the free boundaries when the Poisson
intensity goes to infinity.

\section{Constrained Dynkin games}
Let $(W_t)_{t\geq 0}$ be a $d$-dimensional standard Brownian motion
defined on a filtered probability space
$(\Omega,\mathcal{F},\mathbbm{F}=\{\mathcal{F}_t\}_{t\geq
0},\mathbbm{P})$ with $\mathbb{F}$ being the minimal augmented
filtration of $W$. Let $\{T_i\}_{i\geq 0}$ be the arrival times of
an independent Poisson process with intensity $\lambda$ and minimal
augmented filtration $\mathbb{H}=\{\mathcal{H}_t\}_{t\geq 0}$.
Denote the smallest filtration generated by $\mathbb{F}$ and
$\mathbb{H}$ as $\mathbb{G}=\{\mathcal{G}_t\}_{t\geq 0}$, i.e.
$\mathcal{G}_t=\mathcal{F}_t\vee\mathcal{H}_t$. Without loss of
generality, we also assume that $T_0=0$ and $T_{\infty}=\infty$.

Let $T$ be a finite $\mathbbm{F}$-stopping time representing the
terminal time of the game, and $\xi$ be an
$\mathcal{F}_T$-measurable random variable representing the
corresponding payoff. Define a random variable $M:\Omega\to
\mathbbm{N}$ such that $T_{M}$ is the next Poisson arrival time
after $T$, i.e. $M(\omega)=\sum_{i\geq
1}i\mathbbm{1}_{\{T_{i-1}(\omega)\leq T(\omega)<T_{i}(\omega)\}}$.

For any integer $i\geq 0$, define the control set
\[\mathcal{R}_{T_i}(\lambda)=\{\mathbbm{G}\mbox{-stopping time }\tau\mbox{ for }\tau(\omega)=T_N(\omega)\mbox{ where }i\leq N\leq M(\omega)\}.\]
The subscript $T_i$ in $\mathcal{R}_{T_i}(\lambda)$ represents the
smallest stopping time that is allowed to choose, and $\lambda$
represents the intensity of the underlying Poisson process.

Consider the following \emph{constrained Dynkin game}, where two
players choose their respective stopping times
$\sigma,\tau\in\mathcal{R}_{T_{1}}(\lambda)$ in order to
minimize/maximize the expected value of the discounted payoff
\begin{equation}
\label{def_payoff} R(\sigma,\tau)=\int_0^{\sigma\wedge\tau\wedge
T}e^{-rs}f_s\,ds+e^{-rT}\xi\mathbbm{1}_{\{\sigma\wedge \tau\geq
T\}}+e^{-r\tau}L_{\tau}\mathbbm{1}_{\{\tau<T,\tau\leq\sigma\}}+e^{-r\sigma}U_{\sigma}\mathbbm{1}_{\{\sigma<T,\sigma<
\tau\}},
\end{equation}
where $r>0$ is the discount rate, and $f$, as a real-valued
$\mathbbm{F}$-progressively measurable process, is the running
payoff. The terminal payoff is $U$ if $\sigma$ happens firstly, {$L$
if $\tau$ happens firstly or $\sigma$ and $\tau$ happen
simultaneously}, and $\xi$ otherwise, where $L$ and $U$ are two
real-valued $\mathbbm{F}$-progressively measurable processes.

Let us define the upper and lower values of the constrained Dynkin
game
\begin{equation}
\label{upperValues}
\overline{v}^{\lambda}=\inf_{\sigma\in\mathcal{R}_{T_1}(\lambda)}\sup_{\tau\in\mathcal{R}_{T_1}(\lambda)}\mathbbm{E}\left[R(\sigma,\tau)\right],
\end{equation}
\begin{equation}
\label{lowerValues}
\underline{v}^{\lambda}=\sup_{\tau\in\mathcal{R}_{T_1}(\lambda)}\inf_{\sigma\in\mathcal{R}_{T_1}(\lambda)}\mathbbm{E}\left[R(\sigma,\tau)\right].
\end{equation}
The game (\ref{upperValues})-(\ref{lowerValues}) is said to have
value $v^{\lambda}$ if
$v^{\lambda}=\overline{v}^{\lambda}=\underline{v}^{\lambda}.$ It is
standard to show that if there exists a saddle point
$(\sigma^*,\tau^*)\in \mathcal{R}_{T_1}(\lambda)\times
\mathcal{R}_{T_1}(\lambda)$ such that
$\mathbbm{E}\left[R(\sigma^*,\tau)\right]\leq
\mathbbm{E}\left[R(\sigma^*,\tau^*)\right]\leq
\mathbbm{E}\left[R(\sigma,\tau^*)\right]$ for every
$(\sigma,\tau)\in \mathcal{R}_{T_1}(\lambda)\times
\mathcal{R}_{T_1}(\lambda)$, then the value of this game exists and
equals to $v^{\lambda}=\mathbbm{E}\left[R(\sigma^*,\tau^*)\right].$

There are two new features of the above constrained Dynkin game.
First, there is a control constraint in the sense that only stopping
at Poisson arrival times is allowed. Second, the players are not
allowed to stop at the initial starting time. Instead, they are only
allowed to stop from the first Poisson time onwards.

We also consider an auxiliary game related to the above constrained
Dyknin game by replacing the control set in
(\ref{upperValues})-(\ref{lowerValues}) with
$\mathcal{R}_{T_0}(\lambda)$, so the players are also allowed to
stop at the initial starting time. That is
\begin{equation}
\label{auxupperValues}
\overline{\hat{v}}^{\lambda}=\inf_{\sigma\in\mathcal{R}_{T_0}(\lambda)}\sup_{\tau\in\mathcal{R}_{T_0}(\lambda)}\mathbbm{E}\left[R(\sigma,\tau)\right],
\end{equation}
\begin{equation}
\label{auxlowerValues}
\underline{\hat{v}}^{\lambda}=\sup_{\tau\in\mathcal{R}_{T_0}(\lambda)}\inf_{\sigma\in\mathcal{R}_{T_0}(\lambda)}\mathbbm{E}\left[R(\sigma,\tau)\right].
\end{equation}
Note that the difference between
(\ref{auxupperValues})-(\ref{auxlowerValues}) and
(\ref{upperValues})-(\ref{lowerValues}) is that the former is
allowed to stop at the initial starting time $T_0=0$, while the
latter not. In other words, the players in
(\ref{auxupperValues})-(\ref{auxlowerValues}) first make their
stopping decisions and then move forward, while in
(\ref{upperValues})-(\ref{lowerValues}) they first move forward and
then make their decisions. We shall show that if the game
(\ref{upperValues})-(\ref{lowerValues}) has value $v^{\lambda}$,
then the value of (\ref{auxupperValues})-(\ref{auxlowerValues}) also
exists and is given by
$\hat{v}^{\lambda}=\min\{{U}_0,\max\{v^{\lambda},{L}_0\}\}$, so the
key is to solve the game (\ref{upperValues})-(\ref{lowerValues}).

\subsection{Main result}

To solve the above constrained Dynkin games,  we introduce the
following BSDE defined on a random horizon $[0,T]$:
\begin{equation}
\label{final_game_value} V_{t\wedge T}^{\lambda}=\xi+\int_{t\wedge
T}^T\left[f_s+\lambda\left(L_s-V_s^{\lambda}\right)^+-\lambda\left(V_s^{\lambda}-U_s\right)^+-rV_s^{\lambda}\,\right]ds-\int_{t\wedge
T}^T Z_s^{\lambda}\,dW_s
\end{equation}
for $t\geq 0$. Note that the above BSDE (\ref{final_game_value}) is
often used to construct the solution of a reflected BSDE with two
reflecting barriers $L$ and $U$ (cf. (\ref{final_game_value_cts})).
Intuitively, when $V^{\lambda}$ falls below $L$ (or goes above $U$),
there will be a penalty $\lambda (L-V^{\lambda})$ (or $\lambda
(V^{\lambda}-U)$) incurred, so BSDE (\ref{final_game_value}) is also
refereed to as the penalized equation.

\begin{assumption}
\label{assumption_1} For $t\in[0,T]$, $L_t\leq U_t$, a.s. Moreover,
(i) when $T$ is an unbounded stopping time, the running {payoff} $f$
and the terminal payoffs $L$, $U$ and $\xi$ are all bounded; (ii)
when $T$ is a bounded stopping time, $f$, $L$, $U$ and $\xi$ are
square-integrable, i.e. $\mathbb{E}[\sup_{0\leq t\leq
T}|X_t|^2]<\infty$ for $X=f,L,U$ and $\xi$.
\end{assumption}

The assumption $L\leq U$ is crucial to the existence of the game
value. On the other hand, the conditions (i) and (ii) are to
guarantee the existence and uniqueness of the solution to BSDE
(\ref{final_game_value}), which will in turn be used to construct
the game value and its associated optimal stopping strategy.

\begin{proposition}
Suppose that Assumption \ref{assumption_1} holds. Then, there exists
a unique solution $(V,Z)$ to BSDE (\ref{final_game_value}).
Moreover, (i) when $T$ is unbounded, $V$ is a bounded and continuous
$\mathbb{F}$-adapted process, and $Z \in
\mathcal{M}^{2}_{loc}(0,T;\mathbbm{R}^d)$, where the latter denotes
the space of all $\mathbb{F}$-progressively measurable processes $Z$
such that
\[||Z||^2_{loc}:=\mathbbm{E}\left[\int_0^{t\wedge T} |Z_s|^2\,ds\right]<\infty\quad\text{for}\ t\geq 0;\]
(ii) when $T$ is bounded, then $V$ is a continuous square-integrable
$\mathbb{F}$-adapted process, and $Z \in
\mathcal{M}^{2}(0,T;\mathbbm{R}^d)$.
\end{proposition}

The proof essentially follows from Theorem 4.1 in \cite{Peng} (for
bounded $T$) and Section 5 in \cite{briand1998stability} (for
unbounded $T$), so we omit its proof and refer to \cite{Peng} and
\cite{briand1998stability} for the details. We are now in a position
to state the main result of this paper.

\begin{theorem}
\label{bigTheorem} Suppose that Assumption \ref{assumption_1} holds.
Let $(V^{\lambda},Z^{\lambda})$ be the unique solution to BSDE
(\ref{final_game_value}). Then, the value of the constrained Dynkin
game (\ref{upperValues})-(\ref{lowerValues}) exists and is given by
$v^{\lambda}=\overline{v}^{\lambda}=\underline{v}^{\lambda}=V_0^{\lambda}.$
The corresponding optimal stopping strategy is given by
\begin{equation}
\label{optimalStop} \left\{\begin{array}{l}
\sigma^*_{T_1}=\inf\{T_N\geq T_1:V_{T_N}^{\lambda}\geq U_{T_N}\}\wedge T_{M};\\
\tau^*_{T_1}=\inf\{T_N\geq T_1:V_{T_N}^{\lambda}\leq L_{T_N}\}\wedge
T_{M}.
\end{array}\right.
\end{equation}
Moreover, the value of the Dynkin game
(\ref{auxupperValues})-(\ref{auxlowerValues}) also exists and is
given by
$\hat{v}^{\lambda}=\min\{{U}_0,\max\{v^{\lambda},{L}_0\}\}$, with
the associated optimal stopping strategy $\sigma_{T_0}^*$ and
$\tau_{T_0}^*$.
\end{theorem}

\subsection{Examples}

Theorem \ref{bigTheorem} solves a wide class of problems in a
unified manner, covering from Markovian to non-Markovian situations
and from finite to infinite horizons. In the one-dimensional
homogenous Markovian setting, there usually exists a threshold
strategy. For this, we will discuss a specific convertible-bond
example in section \ref{pcb_poisson_section}. In the rest of the
section, we list several path-dependent examples, which are
difficult to dealt with under Markovian framework (at least it needs
a case-by-case study) but covered by Theorem \ref{bigTheorem}.

(i) \emph{Path-dependent payoffs $L$ and $U$.} Let $T$ be fixed so
it is a constant stopping time and $S$ be a one-dimensional positive
diffusion process adapted to $\mathbb{F}$. For $\delta>0$, consider
an Israeli option written on $S$ with maturity $T$, where the holder
may exercise to get a normal claim but the writer is punished by an
amount $\delta S$ for annulling the contract early (see
\cite{Kifer}). The payoffs $L$ and $U$ may take the form
$L_t=\max\{m,S_t^*\}$ and $U_t=\max\{m,S^*_t\}+\delta S_t$ for
$m>S_0$ and $S^{*}_t=\sup_{0\leq u\leq t}S_u$. This is so called
Israeli Russian option. For $L_t=\int_0^tS_udu$ and
$U_t=\int_0^tS_udu+\delta S_t$, it is called Israeli integral option
(see \cite{Baurdoux}). Under mild integrability assumption on $S$ as
in Assumption \ref{assumption_1}, Theorem \ref{bigTheorem} shows
that the values of both Israeli options exist and the associated
optimal strategies can be characterized via the solution to
(\ref{final_game_value}).

(ii) \emph{Path-dependent stopping time $T$.} Stopping times are
widely used in insurance as indicators of a variety of risks. Let
$S$ be a one-dimensional positive diffusion process adapted to
$\mathbb{F}$. We may consider the following stopping times as the
terminal time of the game: drawdown stopping time $T=\inf\{t\geq 0:
S_t^*-S_t\geq m\}$ for $m\geq 0$; occupation stopping time
$T=\inf\{t\geq m: \int_0^t1_{\{S_u\in A\}}du\geq m\}$ for
$A\subset\mathbb{R}_+$. Note that unlike the standard
first-passage-time (see $\theta^{\lambda}$ in section
\ref{pcb_poisson_section}), both types of path-dependent stopping
times need tailor-made analysis under Markovian framework, but can
be covered by Theorem \ref{bigTheorem} in a unified manner.

\section{Proof of Theorem \ref{bigTheorem}}
We first give an equivalent formulation of the constrained Dynkin
game (\ref{upperValues})-(\ref{lowerValues}). Given the arrival time
$T_i$, define pre-$T_i$ $\sigma$-field
\[\mathcal{G}_{T_i}=\left\{A\in \bigvee_{s\geq 0}\mathcal{G}_s:A\cap \{T_i\leq s\}\in \mathcal{G}_s\mbox{ for }s\geq 0\right\}\]
and $\tilde{\mathbbm{G}}=\{\mathcal{G}_{T_i}\}_{i\geq 0}$. It is
obvious that the upper and lower values of the constrained Dynkin
game can be rewritten as
\begin{equation}
\label{upperValues2} \overline{v}^{\lambda}=\inf_{N^{\sigma}\in
\mathcal{N}_1(\lambda)}\sup_{N^{\tau}\in
\mathcal{N}_1(\lambda)}\mathbbm{E}\left[R(T_{N^{\sigma}},T_{N^{\tau}})\right],
\end{equation}
\begin{equation}
\label{lowerValues2} \underline{v}^{\lambda}=\sup_{N^{\tau}\in
\mathcal{N}_1(\lambda)}\inf_{N^{\sigma}\in
\mathcal{N}_1(\lambda)}\mathbbm{E}\left[R(T_{N^{\sigma}},T_{N^{\tau}})\right],
\end{equation}
where
\[\mathcal{N}_n(\lambda)=\left\{\tilde{\mathbbm{G}}\mbox{-stopping time }N\mbox{ for }n\leq N(\omega)\leq M(\omega)\right\}.\]
The subscript $n$ in $\mathcal{N}_n(\lambda)$ represents the
smallest stopping time that is allowed to choose, and $\lambda$
represents the intensity of the underlying filtration
$\tilde{\mathbbm{G}}$. Both players are allowed to stop at a
sequence of integers $n,n+1,\cdots,M$.

We also observe that a pair of processes
$\left(V^{\lambda},Z^{\lambda}\right)$ solve
(\ref{final_game_value}), if and only if the corresponding
discounted processes
$(Q^{\lambda}_t,\tilde{Z}^{\lambda}_t)=(e^{-rt}V^{\lambda}_t,e^{-rt}Z^{\lambda}_t)$,
for $t\in[0,T]$, solve the following BSDE
\begin{equation}
\label{final_game_value_discounted} Q_{t\wedge
T}^{\lambda}=\tilde{\xi}+\int_{t\wedge T}^T\left[\tilde{f}_s+\lambda
\left(\tilde{L}_s-Q_s^{\lambda}\right)^+-\lambda
\left(Q_s^{\lambda}-\tilde{U}_s\right)^+\right]ds-\int_{t\wedge T}^T
\tilde{Z}_s^{\lambda}\,dW_s,
\end{equation}
where $\tilde{\xi}=e^{-rT}\xi$ and $\tilde{\phi}_s=e^{-rs}\phi_s$
for $\phi=f,L,U$.

Thus, to prove Theorem \ref{bigTheorem}, it is equivalent to show
that $Q_0^{\lambda}=\overline{q}^{\lambda}=\underline{q}^{\lambda}$,
where
\begin{equation}
\label{upperValues_discounted}
\overline{q}^{\lambda}:=\inf_{N^{\sigma}\in
\mathcal{N}_1(\lambda)}\sup_{N^{\tau}\in
\mathcal{N}_1(\lambda)}\mathbbm{E}\left[\tilde{R}\left(T_{N^{\sigma}},T_{N^{\tau}}\right)\right],
\end{equation}
\begin{equation}
\label{lowerValues_discounted}
\underline{q}^{\lambda}:=\sup_{N^{\tau}\in
\mathcal{N}_1(\lambda)}\inf_{N^{\sigma}\in
\mathcal{N}_1(\lambda)}\mathbbm{E}\left[\tilde{R}\left(T_{N^{\sigma}},T_{N^{\tau}}\right)\right],
\end{equation}
with
\[\tilde{R}(\sigma,\tau)=\int_0^{\sigma\wedge\tau\wedge T}\tilde{f}_s\,ds+\tilde{\xi}\mathbbm{1}_{\{\sigma\wedge \tau\geq T\}}+\tilde{L}_{\tau}\mathbbm{1}_{\{\tau<T,\tau\leq\sigma\}}+\tilde{U}_{\sigma}\mathbbm{1}_{\{\sigma<T,\sigma< \tau\}},\]
and the optimal stopping strategy is given by
\begin{equation}
\label{optimalStop_discounted} \left\{\begin{array}{l}
N^{\sigma,*}_1=\inf\{N\geq 1:Q^{\lambda}_{T_N}\geq \tilde{U}_{T_N}\}\wedge M,\\
N^{\tau,*}_1=\inf\{N\geq 1:Q^{\lambda}_{T_N}\leq
\tilde{L}_{T_N}\}\wedge M.
\end{array}\right.
\end{equation}

To prove the above assertions, we start with the following lemma.

\begin{lemma}
\label{firstLemma} Suppose that Assumption \ref{assumption_1} holds.
Then, for any $1\leq n \leq M$, the solution of BSDE
(\ref{final_game_value_discounted}) at time $T_{n-1}$ is the unique
solution of the recursive equation
\begin{align}\label{recursiveEq}
&Q_{T_{n-1}}^{\lambda}=\mathbbm{E}\left[\int_{T_{n-1}}^{T_n\wedge T}\tilde{f}_s\,ds+\tilde{\xi}\mathbbm{1}_{\{T_n>T\}}\right.\\
&+\left.\left.\left(\mathbbm{1}_{\{Q_{T_n}^{\lambda}\geq
\tilde{U}_{T_n}\}}\tilde{U}_{T_n}+\mathbbm{1}_{\{Q_{T_n}^{\lambda}\leq
\tilde{L}_{T_n}\}}\tilde{L}_{T_n}+\mathbbm{1}_{\{\tilde{L}_{T_n}<Q_{T_n}^{\lambda}<\tilde{U}_{T_n}\}}Q_{T_n}^{\lambda}\right)\mathbbm{1}_{\{T_n\leq
T\}}\right|\mathcal{G}_{T_{n-1}}\right].\notag
\end{align}
\end{lemma}
\begin{proof}
Applying It\^{o}'s formula to $\alpha_tQ_t^{\lambda}$, where
$\alpha_t=e^{-\lambda t}$, we obtain, for $t\in[0,T]$,
\[\alpha_t Q_t^{\lambda}=\alpha_T Q_T^{\lambda}+\int_t^T \alpha_s\left[\tilde{f}_s+\lambda F_s(Q_s^{\lambda})\right]\,ds-\int_t^T \alpha_s
\tilde{Z}_s^{\lambda}\,dW_s,\] where
$F_s(Q_s^{\lambda}):=Q_s^{\lambda}+(\tilde{L}_s-Q_s^{\lambda})^+-(Q_s^{\lambda}-\tilde{U}_s)^+$.
Consequently,
\begin{eqnarray*}
&&Q_{T_{n-1}}^{\lambda}\\
&=&\frac{\alpha_T}{\alpha_{T_{n-1}}}\tilde{\xi}+\int_{T_{n-1}}^T
\frac{\alpha_s}{\alpha_{T_{n-1}}}\left[\tilde{f}_s+\lambda
F_s(Q_s^{\lambda})\right]\,ds-\int_{T_{n-1}}^T \frac{\alpha_s}{\alpha_{T_{n-1}}}\tilde{Z}_s^{\lambda}\,dW_s\\
&=&\mathbbm{E}\left[\left.e^{-\lambda
(T-T_{n-1})}\tilde{\xi}+\int_{T_{n-1}}^T e^{-\lambda
(s-T_{n-1})}\left[\tilde{f}_s+\lambda
F_s(Q_s^{\lambda})\right]\,ds\right|\mathcal{G}_{T_{n-1}}\right].
\end{eqnarray*}

On the other hand, we use the conditional density $\lambda
e^{-\lambda (x-T_{n-1})}\,dx$ of $T_n$ to calculate the right-hand
side of (\ref{recursiveEq}):
\begin{eqnarray*}
&&\mathbbm{E}\left[\left.\int_{T_{n-1}}^{T_n\wedge T}\tilde{f}_s\,ds\right|\mathcal{G}_{T_{n-1}}\right]\\
&=&\mathbbm{E}\left[\left.e^{-\lambda (T-{T_{n-1}})}\int_{T_{n-1}}^T\tilde{f}_s\,ds+\int_{T_{n-1}}^T\lambda e^{-\lambda (x-{T_{n-1}})}\int_{T_{n-1}}^x\tilde{f}_s\,ds\,dx\right|\mathcal{G}_{T_{n-1}}\right]\\
&=&\mathbbm{E}\left[\left.e^{-\lambda (T-{T_{n-1}})}\int_{T_{n-1}}^T\tilde{f}_s\,ds+\int_{T_{n-1}}^T\tilde{f}_s\int_s^T\lambda e^{-\lambda (x-{T_{n-1}})}\,dx\,ds\right|\mathcal{G}_{T_{n-1}}\right]\\
&=&\mathbbm{E}\left[\left.\int_{T_{n-1}}^Te^{-\lambda(s-{T_{n-1}})}\tilde{f}_s\,ds\right|\mathcal{G}_{T_{n-1}}\right],
\end{eqnarray*}
where we used integration by parts in the second equality.
Similarly, we have
$$\mathbbm{E}\left[\left.\tilde{\xi}\mathbbm{1}_{\{T_n>T\}}\right|\mathcal{G}_{T_{n-1}}\right]=\mathbbm{E}\left[\left.e^{-\lambda
(T-T_{n-1})}\tilde{\xi}\right|\mathcal{G}_{T_{n-1}}\right],$$ and
\begin{align*}
&\mathbbm{E}\left[\left.\left(\mathbbm{1}_{\{Q_{T_n}^{\lambda}\geq \tilde{U}_{T_n}\}}\tilde{U}_{T_n}+\mathbbm{1}_{\{Q_{T_n}^{\lambda}\leq \tilde{L}_{T_n}\}}\tilde{L}_{T_n}+\mathbbm{1}_{\{\tilde{L}_{T_n}<Q_{T_n}^{\lambda}<\tilde{U}_{T_n}\}}Q_{T_n}^{\lambda}\right)\mathbbm{1}_{\{T_n\leq T\}}\right|\mathcal{G}_{T_{n-1}}\right]\\
=&\ \mathbbm{E}\left[\left.\int_{T_{n-1}}^T\lambda e^{-\lambda
(s-T_{n-1})}\left( \mathbbm{1}_{\{Q_s^{\lambda}\geq
\tilde{U}_s\}}\tilde{U}_s+\mathbbm{1}_{\{Q_s^{\lambda}\leq
\tilde{L}_s\}}\tilde{L}_s+\mathbbm{1}_{\{\tilde{L}_s<Q_s^{\lambda}<\tilde{U}_s\}}Q_s^{\lambda}\right)\,ds\right|\mathcal{G}_{T_{n-1}}\right].
\end{align*}
It follows that (\ref{recursiveEq}) holds. Since the recursive
equation (\ref{recursiveEq}) obviously admits a unique solution,
$Q_{T_{n-1}}^{\lambda}$ is then the unique solution of
(\ref{recursiveEq}) for $1\leq n\leq M$.
\end{proof}

As a direct consequence of Lemma \ref{firstLemma}, if we define
$\hat{Q}^{\lambda}=\min\{\tilde{U},\max\{Q^{\lambda},\tilde{L}\}\},$
then by the assumption $L\leq U$ (so $\tilde{L}\leq \tilde{U}$),
\[\hat{Q}^{\lambda}=
\mathbbm{1}_{\{Q^{\lambda}\geq
\tilde{U}\}}\tilde{U}+\mathbbm{1}_{\{Q^{\lambda}\leq
\tilde{L}\}}\tilde{L}+\mathbbm{1}_{\{\tilde{L}<Q^{\lambda}<\tilde{U}\}}Q^{\lambda},\]
and thus, $\hat{Q}^{\lambda}$ satisfies the following recursive
equation: For $1\leq n\leq M$,
\begin{align}
\label{recursiveEq2}
&\ \hat{Q}_{T_{n-1}}^{\lambda}\\
=&\
\min\left\{\tilde{U}_{T_{n-1}},\max\left\{\mathbbm{E}\left[\left.\int_{T_{n-1}}^{T_n\wedge
T}\tilde{f}_s\,ds+\tilde{\xi}\mathbbm{1}_{\{T_n>T\}}+\hat{Q}_{T_{n}}^{\lambda}\mathbbm{1}_{\{T_n\leq
T\}}\right|\mathcal{G}_{T_{n-1}}\right],\tilde{L}_{T_{n-1}}\right\}\right\},\notag
\end{align}
which also admits a unique solution since we can calculate its
solution backwards in a recursive way. We show that
$\hat{Q}^{\lambda}_{T_{n-1}}$ is the value of another constrained
Dynkin game. Introduce the upper and lower values of an auxiliary
constrained Dynkin game as
\begin{equation}
\label{aux_upperValues2}
\overline{\hat{q}}_{T_{n-1}}^{\lambda}=\essinf_{N^{\sigma}\in
\mathcal{N}_{n-1}(\lambda)}\esssup_{N^{\tau}\in
\mathcal{N}_{n-1}(\lambda)}\mathbbm{E}\left[\tilde{R}_{n-1}(T_{N^{\sigma}},T_{N^{\tau}})|\mathcal{G}_{T_{n-1}}\right],
\end{equation}
\begin{equation}
\label{aux_lowerValues2}
\underline{\hat{q}}_{T_{n-1}}^{\lambda}=\esssup_{N^{\tau}\in
\mathcal{N}_{n-1}(\lambda)}\essinf_{N^{\sigma}\in
\mathcal{N}_{n-1}(\lambda)}\mathbbm{E}\left[\tilde{R}_{n-1}(T_{N^{\sigma}},T_{N^{\tau}})|\mathcal{G}_{T_{n-1}}\right],
\end{equation}
where
\[\tilde{R}_{n-1}(\sigma,\tau)=\int_{T_{n-1}\wedge T}^{\sigma\wedge\tau\wedge T}\tilde{f}_s\,ds+\tilde{\xi}\mathbbm{1}_{\{\sigma\wedge \tau\geq T\}}+\tilde{L}_{\tau}\mathbbm{1}_{\{\tau<T,\tau\leq\sigma\}}+\tilde{U}_{\sigma}\mathbbm{1}_{\{\sigma<T,\sigma< \tau\}}\]
with $\tilde{R}_{0}(\sigma,\tau)=\tilde{R}(\sigma,\tau)$, and
\[\mathcal{N}_{n-1}(\lambda)=\left\{\tilde{\mathbbm{G}}\mbox{-stopping time }N\mbox{ for }n-1\leq N(\omega)\leq M(\omega)\right\}.\]

Note that, when $n=1$,
(\ref{aux_upperValues2})-(\ref{aux_lowerValues2}) corresponds to the
auxiliary Dynkin game (\ref{auxupperValues})-(\ref{auxlowerValues}),
which will be used to solve the original constrained Dynkin game.
The difference between the auxiliary game and the original game is
that the players first make their stopping decisions and then move
forward in the auxiliary game, while in original game they first
move forward and then make their decisions.

\begin{lemma}
\label{secondLemma} Suppose that Assumption \ref{assumption_1}
holds. Then, for any $1\leq n\leq M$, the value of the auxiliary
{constrained} Dynkin game
(\ref{aux_upperValues2})-(\ref{aux_lowerValues2}) exists. Its value,
denoted by $\hat{q}_{T_{n-1}}^{\lambda}$, satisfies the recursive
equation (\ref{recursiveEq2}), namely,
\begin{align*}
&\hat{q}_{T_{n-1}}^{\lambda}\\
=&\
\min\left\{\tilde{U}_{T_{n-1}},\max\left\{\mathbbm{E}\left[\left.\int_{T_{n-1}}^{T_n\wedge
T}\tilde{f}_s\,ds+\tilde{\xi}\mathbbm{1}_{\{T_n>T\}}+\hat{q}_{T_{n}}^{\lambda}\mathbbm{1}_{\{T_n\leq
T\}}\right|\mathcal{G}_{T_{n-1}}\right],\tilde{L}_{T_{n-1}}\right\}\right\}.
\end{align*}
Hence, $\hat{q}^{\lambda}_{T_{n-1}}=\hat{Q}^{\lambda}_{T_{n-1}}$
a.s. The optimal stopping strategy of
(\ref{aux_upperValues2})-(\ref{aux_lowerValues2}) is given by
\begin{equation}
\label{optimalStop_aux} \left\{\begin{array}{l}
\hat{N}^{\sigma,*}_{n-1}=\inf\{N\geq n-1:\hat{q}_{T_N}^{\lambda}= \tilde{U}_{T_N}\}\wedge M;\\
\hat{N}^{\tau,*}_{n-1}=\inf\{N\geq n-1:\hat{q}_{T_N}^{\lambda}=
\tilde{L}_{T_N}\}\wedge M.
\end{array}\right.
\end{equation}
\end{lemma}
\begin{proof}
Without loss of generality, we may assume $\tilde{f}_s=0$.

\noindent\emph{Step 1.} Since $T_{M-1}\leq T<T_{M}$, the upper value
of the auxiliary game (\ref{aux_upperValues2}) is equivalent to
\begin{align*}
\overline{\hat{q}}_{T_{n-1}}^{\lambda}=\essinf_{N^{\sigma}\in \mathcal{N}_{n-1}(\lambda)}\esssup_{N^{\tau}\in \mathcal{N}_{n-1}(\lambda)}\mathbbm{E}&\left[\tilde{\xi}\mathbbm{1}_{\{N^{\sigma}=N^{\tau}=M\}}+\tilde{L}_{T_{N^{\tau}}}\mathbbm{1}_{\{n-1\leq N^{\tau}\leq M-1,N^{\tau}\leq N^{\sigma}\}}\right.\\
+&\left.\tilde{U}_{T_{N^{\sigma}}}\mathbbm{1}_{\{n-1\leq
N^{\sigma}\leq M-1,N^{\sigma}<
N^{\tau}\}}|\mathcal{G}_{T_{n-1}}\right].
\end{align*}
We claim that
\begin{equation}
\label{hatvT_M-1}
\overline{\hat{q}}_{T_{M-1}}^{\lambda}=\min\left\{\tilde{U}_{T_{M-1}},\max\left\{\mathbbm{E}\left[\tilde{\xi}|\mathcal{G}_{T_{M-1}}\right],\tilde{L}_{T_{M-1}}\right\}\right\},
\end{equation}
and, for $n-1\leq i\leq M-2$,
\begin{equation}
\label{hatvT_i}
\overline{\hat{q}}_{T_i}^{\lambda}=\min\left\{\tilde{U}_{T_i},\max\left\{\mathbbm{E}\left[\overline{\hat{q}}_{T_{i+1}}^{\lambda}|\mathcal{G}_{T_i}\right],\tilde{L}_{T_i}\right\}\right\}.
\end{equation}
If (\ref{hatvT_M-1})-(\ref{hatvT_i}) hold, then
\begin{eqnarray*}
\overline{\hat{q}}_{T_{n-1}}^{\lambda}&=&\min\left\{\tilde{U}_{T_{n-1}},\max\left\{\mathbbm{E}\left[\tilde{\xi}\mathbbm{1}_{\{n=M\}}+\overline{\hat{q}}_{T_{n}}^{\lambda}\mathbbm{1}_{\{n\leq M-1\}}|\mathcal{G}_{T_{n-1}}\right],\tilde{L}_{T_{n-1}}\right\}\right\}\\
&=&\min\left\{\tilde{U}_{T_{n-1}},\max\left\{\mathbbm{E}\left[\tilde{\xi}\mathbbm{1}_{\{T_n>
T\}}+\overline{\hat{q}}_{T_{n}}^{\lambda}\mathbbm{1}_{\{T_n\leq
T\}}|\mathcal{G}_{T_{n-1}}\right],\tilde{L}_{T_{n-1}}\right\}\right\},
\end{eqnarray*}
which is the recursive equation (\ref{recursiveEq2}).

Similarly, we obtain that $\underline{\hat{q}}_{T_{n-1}}^{\lambda}$
also satisfies the recursive equation (\ref{recursiveEq2}). Since
(\ref{recursiveEq2}) admits a unique solution, it is clear that
$\overline{\hat{q}}_{T_{n-1}}^{\lambda}=\underline{\hat{q}}_{T_{n-1}}^{\lambda}=\hat{q}_{T_{n-1}}^{\lambda}=\hat{Q}^{\lambda}_{T_{n-1}}$
a.s.

\noindent\emph{Step 2.} Next, we show
(\ref{hatvT_M-1})-(\ref{hatvT_i}). Indeed, for $i=M-1$,
\begin{eqnarray*}
\overline{\hat{q}}_{T_{M-1}}^{\lambda}&=&\essinf_{N^{\sigma}\in \mathcal{N}_{M-1}(\lambda)}\esssup_{N^{\tau}\in \mathcal{N}_{M-1}(\lambda)}\mathbbm{E}\left[\tilde{\xi}\mathbbm{1}_{\{N^{\sigma}=N^{\tau}=M\}}+\tilde{L}_{T_{M-1}}\mathbbm{1}_{\{M-1=N^{\tau}\leq N^{\sigma}\}}\right.\\
&&+\left.\tilde{U}_{T_{M-1}}\mathbbm{1}_{\{M-1=N^{\sigma}< N^{\tau}\}}|\mathcal{G}_{T_{M-1}}\right]\\
&=&\min_{N^{\sigma}\in \mathcal{N}_{M-1}(\lambda)}\max_{N^{\tau}\in
\mathcal{N}_{M-1}(\lambda)}
\left\{\mathbbm{E}[\tilde{\xi}|\mathcal{G}_{T_{M-1}}]\mathbbm{1}_{\{N^{\sigma}=N^{\tau}=
M\}}+\tilde{L}_{T_{M-1}}\mathbbm{1}_{\{M-1=N^{\tau}\leq
N^{\sigma}\}}\right.\\
&&+\left.\tilde{U}_{T_{M-1}}\mathbbm{1}_{\{M-1=N^{\sigma}<N^{\tau}\}}\right\}\\
&=&\min\left\{\tilde{U}_{T_{M-1}},\max\left\{\mathbbm{E}[\tilde{\xi}|\mathcal{G}_{T_{M-1}}],\tilde{L}_{T_{M-1}}\right\}\right\}.
\end{eqnarray*}
In general, for $n-1\leq i\leq M-2$, we have
\begin{eqnarray*}
\overline{\hat{q}}_{T_i}^{\lambda}&=&\essinf_{N^{\sigma}\in \mathcal{N}_i(\lambda)}\esssup_{N^{\tau}\in \mathcal{N}_i(\lambda)}\mathbbm{E}\left[\tilde{\xi}\mathbbm{1}_{\{N^{\sigma}=N^{\tau}=M\}}+\tilde{L}_{T_{N^{\tau}}}\mathbbm{1}_{\{i\leq N^{\tau}\leq M-1,N^{\tau}\leq N^{\sigma}\}}\right.\\
&&+\left.\tilde{U}_{T_{N^{\sigma}}}\mathbbm{1}_{\{i\leq
N^{\sigma}\leq M-1,N^{\sigma}< N^{\tau}\}}|\mathcal{G}_{T_i}\right].
\end{eqnarray*}
Taking conditional expectation on $\mathcal{G}_{T_{i+1}}$ further
yields
\begin{eqnarray*}
\overline{\hat{q}}_{T_i}^{\lambda}&=&\essinf_{N^{\sigma}\in \mathcal{N}_i(\lambda)}\esssup_{N^{\tau}\in \mathcal{N}_i(\lambda)}\mathbbm{E}\left[\tilde{L}_{T_i}\mathbbm{1}_{\{i=N^{\tau}\leq N^{\sigma}\}}+\tilde{U}_{T_i}\mathbbm{1}_{\{i=N^{\sigma}< N^{\tau}\}}\right.\\
&&+\ \mathbbm{E}\left[\tilde{\xi}\mathbbm{1}_{\{N^{\sigma}=N^{\tau}=M\}}+\tilde{L}_{T_{N^{\tau}}}\mathbbm{1}_{\{i+1\leq N^{\tau}\leq M-1,N^{\tau}\leq N^{\sigma}\}}\right.\\
&&+\left.\left.\tilde{U}_{T_{N^{\sigma}}}\mathbbm{1}_{\{i+1\leq N^{\sigma}\leq M-1,N^{\sigma}< N^{\tau}\}}|\mathcal{G}_{T_{i+1}}\right]|\mathcal{G}_{T_i}\right]\\
&=&\min\left\{\tilde{U}_{T_i},\max\left\{\mathbbm{E}\left[\overline{\hat{q}}_{T_{i+1}}^{\lambda}|\mathcal{G}_{T_i}\right],\tilde{L}_{T_i}\right\}\right\},
\end{eqnarray*}
where the second equality holds since the operations
$\essinf_{N^{\sigma}\in
\mathcal{N}_{i+1}(\lambda)}\esssup_{N^{\tau}\in
\mathcal{N}_{i+1}(\lambda)}$ and
$\mathbbm{E}\left[\cdot|\mathcal{G}_{T_i}\right]$ are
interchangeable, which will be proved in the next step.

\noindent\emph{Step 3.} In this step, we show the operations
$\essinf_{N^{\sigma}\in
\mathcal{N}_{i+1}(\lambda)}\esssup_{N^{\tau}\in
\mathcal{N}_{i+1}(\lambda)}$ and
$\mathbbm{E}\left[\cdot|\mathcal{G}_{T_i}\right]$ are
interchangeable, i.e. (\ref{interchange}) below holds. To this end,
for fixed $i$ and $N^{\sigma}\in\mathcal{N}_i(\lambda)$, we note
that the family
\begin{equation}
\label{increasing_directed}
\left(\mathbbm{E}\left[\tilde{R}_i(T_{N^{\sigma}},T_{N^{\tau}})|\mathcal{G}_{T_i}\right],N^{\tau}\in\mathcal{N}_i(\lambda)\right)
\end{equation}
is an increasing directed set. Indeed, if we choose arbitrary
$N^{\tau}_1,N^{\tau}_2\in\mathcal{N}_i(\lambda)$ and let
$X_j=\mathbbm{E}\left[\tilde{R}_i(T_{N^{\sigma}},T_{N^{\tau}_j})|\mathcal{G}_{T_i}\right],$
for $j=1,2$. Then, defining the stopping time $N^{\tau}$ as
$N^{\tau}=N^{\tau}_1\mathbbm{1}_{\{X_1\geq
X_2\}}+N^{\tau}_2\mathbbm{1}_{\{X_1<X_2\}},$ we have
$N^{\tau}\in\mathcal{N}_i(\lambda)$ and
$\mathbbm{E}\left[\tilde{R}_i(T_{N^{\sigma}},T_{N^{\tau}})|\mathcal{G}_{T_i}\right]\geq
\max\{X_1,X_2\}$.

Similarly, we also have, for fixed $i$, the family
\begin{equation}
\label{decreasing_directed} \left(\esssup_{N^{\tau}\in
\mathcal{N}_i(\lambda)}\mathbbm{E}\left[\tilde{R}_i(T_{N^{\sigma}},T_{N^{\tau}})|\mathcal{G}_{T_i}\right],N^{\sigma}\in\mathcal{N}_i(\lambda)\right)
\end{equation}
is a decreasing directed set. Under Assumption \ref{assumption_1},
it is obvious that both (\ref{increasing_directed}) and
(\ref{decreasing_directed}) are uniformly integrable. Therefore, by
Proposition VI-1-1 of Neveu \cite{neveu1975discrete}, we obtain
\begin{eqnarray}\label{interchange}
\mathbbm{E}\left[\left.\overline{\hat{q}}_{T_{i+1}}^{\lambda}\right|\mathcal{G}_{T_i}\right]&=&\mathbbm{E}\left[\left.\essinf_{N^{\sigma}\in \mathcal{N}_{i+1}(\lambda)}\esssup_{N^{\tau}\in \mathcal{N}_{i+1}(\lambda)}\mathbbm{E}\left[\tilde{R}_{i+1}(T_{N^{\sigma}},T_{N^{\tau}})|\mathcal{G}_{T_{i+1}}\right]\right|\mathcal{G}_{T_i}\right]\notag\\
&=&{\essinf_{N^{\sigma}\in \mathcal{N}_{i+1}(\lambda)}\mathbbm{E}\left[\left.\esssup_{N^{\tau}\in \mathcal{N}_{i+1}(\lambda)}\mathbbm{E}\left[\tilde{R}_{i+1}(T_{N^{\sigma}},T_{N^{\tau}})|\mathcal{G}_{T_{i+1}}\right]\right|\mathcal{G}_{T_i}\right]}\notag\\
&=&\essinf_{N^{\sigma}\in
\mathcal{N}_{i+1}(\lambda)}\esssup_{N^{\tau}\in\mathcal{N}_{i+1}(\lambda)}\mathbbm{E}\left[\tilde{R}_{i+1}(T_{N^{\sigma}},T_{N^{\tau}})|\mathcal{G}_{T_i}\right].
\end{eqnarray}

\noindent\emph{Step 4.} It remains to prove that
$\left(\hat{N}^{\sigma,*}_{n-1},\hat{N}^{\tau,*}_{n-1}\right)$ in
(\ref{optimalStop_aux}) are indeed the optimal stopping times for
the auxiliary Dynkin game
(\ref{aux_upperValues2})-(\ref{aux_lowerValues2}), i.e. for every
$(N^{\sigma},N^{\tau})\in\mathcal{N}_{n-1}(\lambda)\times
\mathcal{N}_{n-1}(\lambda)$,
\begin{eqnarray*}
\mathbbm{E}\left[\tilde{R}_{n-1}\left(T_{\hat{N}^{\sigma,*}_{n-1}},T_{N^{\tau}}\right)|\mathcal{G}_{T_{n-1}}\right]&\leq & \mathbbm{E}\left[\tilde{R}_{n-1}\left(T_{\hat{N}^{\sigma,*}_{n-1}},T_{\hat{N}^{\tau,*}_{n-1}}\right)|\mathcal{G}_{T_{n-1}}\right]\label{step4_optimal_1}\\
&\leq &
\mathbbm{E}\left[\tilde{R}_{n-1}\left(T_{N^{\sigma}},T_{\hat{N}^{\tau,*}_{n-1}}\right)|\mathcal{G}_{T_{n-1}}\right].\label{step4_optimal_2}
\end{eqnarray*}

To this end, it suffices to prove that

\noindent(i) $\left(\hat{q}_{T_{m\wedge
\hat{N}^{\sigma,*}_{n-1}\wedge
\hat{N}^{\tau,*}_{n-1}}}^{\lambda}\right)_{m\geq n-1}$ is a
$\tilde{\mathbbm{G}}$-martingale;

\noindent(ii) $\left(\hat{q}_{T_{m\wedge
\hat{N}^{\sigma,*}_{n-1}\wedge N^{\tau}}}^{\lambda}\right)_{m\geq
n-1}$ is a $\tilde{\mathbbm{G}}$-supermartingale for any
$N^{\tau}\in\mathcal{N}_{n-1}(\lambda)$;

\noindent(iii) $\left(\hat{q}_{T_{m\wedge N^{\sigma}\wedge
\hat{N}^{\tau,*}_{n-1}}}^{\lambda}\right)_{m\geq n-1}$ is a
$\tilde{\mathbbm{G}}$-submartingale for any
$N^{\sigma}\in\mathcal{N}_{n-1}(\lambda)$.

Indeed, we have
\begin{align*}
&\mathbbm{E}\left[\left.\hat{q}_{T_{(m+1)\wedge \hat{N}^{\sigma,*}_{n-1}\wedge \hat{N}^{\tau,*}_{n-1}}}^{\lambda}\right|\mathcal{G}_{T_m}\right]\\
=\ &\mathbbm{E}\left[\left.\left(\sum_{j=n-1}^m\mathbbm{1}_{\{\hat{N}^{\sigma,*}_{n-1}\wedge \hat{N}^{\tau,*}_{n-1}=j\}}+\mathbbm{1}_{\{\hat{N}^{\sigma,*}_{n-1}\wedge \hat{N}^{\tau,*}_{n-1}\geq m+1\}}\right)\hat{q}_{T_{(m+1)\wedge \hat{N}^{\sigma,*}_{n-1}\wedge \hat{N}^{\tau,*}_{n-1}}}^{\lambda}\right|\mathcal{G}_{T_m}\right]\\
=&\sum_{j=n-1}^m\mathbbm{1}_{\{\hat{N}^{\sigma,*}_{n-1}\wedge \hat{N}^{\tau,*}_{n-1}=j\}}\hat{q}_{T_j}^{\lambda}+\mathbbm{1}_{\{\hat{N}^{\sigma,*}_{n-1}\wedge \hat{N}^{\tau,*}_{n-1}\geq m+1\}}\mathbbm{E}\left[\left.\hat{q}_{T_{m+1}}^{\lambda}\right|\mathcal{G}_{T_m}\right]\\
=&\sum_{j=n-1}^m\mathbbm{1}_{\{\hat{N}^{\sigma,*}_{n-1}\wedge
\hat{N}^{\tau,*}_{n-1}=j\}}\hat{q}_{T_j}^{\lambda}+\mathbbm{1}_{\{\hat{N}^{\sigma,*}_{n-1}\wedge
\hat{N}^{\tau,*}_{n-1}\geq
m+1\}}\hat{q}_{T_{m}}^{\lambda}=\hat{q}_{T_{m\wedge
\hat{N}^{\sigma,*}_{n-1}\wedge \hat{N}^{\tau,*}_{n-1}}}^{\lambda},
\end{align*}
where the second last equality follows from the definition of
$\left(\hat{N}^{\sigma,*}_{n-1},\hat{N}^{\tau,*}_{n-1}\right)$ in
(\ref{optimalStop_aux}), so the martingale property (i) has been
proved.

To prove the supermartingale property (ii), we note that
\begin{align*}
&\ \mathbbm{E}\left[\left.\hat{q}_{T_{(m+1)\wedge \hat{N}^{\sigma,*}_{n-1}\wedge N^{\tau}}}^{\lambda}\right|\mathcal{G}_{T_m}\right]\\
=&\ \mathbbm{E}\left[\left.\hat{q}_{T_{(m+1)\wedge \hat{N}^{\sigma,*}_{n-1}}}^{\lambda}\mathbbm{1}_{\{N^{\tau}\geq m+1\}}+\hat{q}_{T_{\hat{N}^{\sigma,*}_{n-1}\wedge N^{\tau}}}^{\lambda}\mathbbm{1}_{\{N^{\tau}\leq m\}}\right|\mathcal{G}_{T_m}\right]\\
=&\
\mathbbm{E}\left[\left.\left(\sum_{j=n-1}^m\mathbbm{1}_{\{\hat{N}^{\sigma,*}_{n-1}=j\}}+\mathbbm{1}_{\{\hat{N}^{\sigma,*}_{n-1}\geq
m+1\}}\right)\hat{q}_{T_{(m+1)\wedge
\hat{N}^{\sigma,*}_{n-1}}}^{\lambda}\mathbbm{1}_{\{N^{\tau}\geq
m+1\}}\right.\right.\\
&\left.\left.\ +\ \hat{q}_{T_{\hat{N}^{\sigma,*}_{n-1}\wedge N^{\tau}}}^{\lambda}\mathbbm{1}_{\{N^{\tau}\leq m\}}\right|\mathcal{G}_{T_m}\right]\\
=&\left(\sum_{j=n-1}^m\mathbbm{1}_{\{\hat{N}^{\sigma,*}_{n-1}=j\}}\hat{q}_{T_{j}}^{\lambda}+\mathbbm{1}_{\{\hat{N}^{\sigma,*}_{n-1}\geq
m+1\}}\mathbbm{E}\left[\left.\hat{q}_{T_{m+1}}^{\lambda}\right|\mathcal{G}_{T_m}\right]\right)\mathbbm{1}_{\{N^{\tau}\geq
m+1\}}\\
&\ +\ \hat{q}_{T_{\hat{N}^{\sigma,*}_{n-1}\wedge
N^{\tau}}}^{\lambda}\mathbbm{1}_{\{N^{\tau}\leq m\}}.
\end{align*}
Using the definition of $\hat{N}^{\sigma,*}_{n-1}$ in
(\ref{optimalStop_aux}), we further have
\[\mathbbm{E}\left[\left.\hat{q}_{T_{m+1}}^{\lambda}\right|\mathcal{G}_{T_m}\right]\leq \max\left\{\mathbbm{E}\left[\left.\hat{q}_{T_{m+1}}^{\lambda}\right|\mathcal{G}_{T_m}\right],\tilde{L}_{T_m}\right\}=\hat{q}_{T_{m}}^{\lambda}
\ \text{on}\ \{\hat{N}^{\sigma,*}_{n-1}\geq  m+1\}.
\]
In turn,
\begin{align*}
&\mathbbm{E}\left[\left.\hat{q}_{T_{(m+1)\wedge \hat{N}^{\sigma,*}_{n-1}\wedge N^{\tau}}}^{\lambda}\right|\mathcal{G}_{T_m}\right]\\
\leq &\left(\sum_{j=n-1}^m\mathbbm{1}_{\{\hat{N}^{\sigma,*}_{n-1}=j\}}\hat{q}_{T_{j}}^{\lambda}+\mathbbm{1}_{\{\hat{N}^{\sigma,*}_{n-1}\geq m+1\}}\hat{q}_{T_{m}}^{\lambda}\right)\mathbbm{1}_{\{N^{\tau}\geq m+1\}}+\hat{q}_{T_{\hat{N}^{\sigma,*}_{n-1}\wedge N^{\tau}}}^{\lambda}\mathbbm{1}_{\{N^{\tau}\leq m\}}\\
=&\ \hat{q}_{T_{m\wedge
\hat{N}^{\sigma,*}_{n-1}}}^{\lambda}\mathbbm{1}_{\{N^{\tau}\geq
m+1\}}+\hat{q}_{T_{\hat{N}^{\sigma,*}_{n-1}\wedge
N^{\tau}}}^{\lambda}\mathbbm{1}_{\{N^{\tau}\leq
m\}}=\hat{q}_{T_{m\wedge \hat{N}^{\sigma,*}_{n-1}\wedge
N^{\tau}}}^{\lambda},
\end{align*}
which proves the supermartingale property (ii). Likewise, the
submartingale property (iii) can be proved in a similar way, and the
proof of the lemma is completed.
\end{proof}

We are now in a position to prove Theorem \ref{bigTheorem}. By
Lemmas \ref{firstLemma} and \ref{secondLemma}, we have
\begin{eqnarray}\label{inequality}
Q_0^{\lambda}&=&\mathbbm{E}\left[\int_0^{T_1\wedge T}\tilde{f}_s\,ds+\tilde{\xi}\mathbbm{1}_{\{T_1>T\}}+\hat{Q}_{T_1}^{\lambda}\mathbbm{1}_{\{T_1\leq T\}}\right]\notag\\
&=&\mathbbm{E}\left[\int_0^{T_1\wedge T}\tilde{f}_s\,ds+\tilde{\xi}\mathbbm{1}_{\{T_1>T\}}+\hat{q}_{T_1}^{\lambda}\mathbbm{1}_{\{T_1\leq T\}}\right]\notag\\
&\geq &\mathbbm{E}\bigg[\int_0^{T_1\wedge
T}\tilde{f}_s\,ds+\tilde{\xi}\mathbbm{1}_{\{T_1>
T\}}+{\mathbbm{E}\left[\tilde{R}_1(T_{\hat{N}^{\sigma,*}_1},T_{N^{\tau}})|\mathcal{G}_{T_1}\right]}\mathbbm{1}_{\{T_1\leq
T\}}\bigg]
\end{eqnarray}
for any $N^{\tau}\in\mathcal{N}_{1}(\lambda)$, where last inequality
follows from the supermartingale property (ii). Moreover, recall
that
\begin{align*}
&\ \mathbbm{E}\left[\tilde{R}_1(T_{\hat{N}^{\sigma,*}_1},T_{N^{\tau}})|\mathcal{G}_{T_1}\right]\\
=&\ \mathbbm{E}\left[\int^{T_{\hat{N}^{\sigma,*}_1}\wedge
T_{N^{\tau}}\wedge T}_{T_1\wedge T}\tilde{f}_s\,ds
+\tilde{\xi}\mathbbm{1}_{\left\{T_{\hat{N}^{\sigma,*}_1}\wedge
T_{N^{\tau}}\geq
T\right\}}\right.\\
&\left.+\
\tilde{L}_{T_{N^{\tau}}}\mathbbm{1}_{\left\{T_{N^{\tau}}<T,T_{N^{\tau}}\leq
T_{\hat{N}^{\sigma,*}_1}
\right\}}+\tilde{U}_{T_{\hat{N}^{\sigma,*}_1}}\mathbbm{1}_{\left\{T_{\hat{N}^{\sigma,*}_1}<T,T_{\hat{N}^{\sigma,*}_1}<
T_{N^{\tau}}\right\}}|\mathcal{G}_{T_1}\right].
\end{align*}
Plugging the above expression into (\ref{inequality}) further yields
\begin{align*}
Q_0^{\lambda}\geq&\ \mathbbm{E}\left[\int^{T_{\hat{N}^{\sigma,*}_1}\wedge T_{N^{\tau}}\wedge T}_{0}\tilde{f}_s\,ds+\tilde{\xi}\mathbbm{1}_{\left\{T_{\hat{N}^{\sigma,*}_1}\wedge T_{N^{\tau}}\geq T\right\}}+\tilde{L}_{T_{N^{\tau}}}\mathbbm{1}_{\left\{T_{N^{\tau}}<T,T_{N^{\tau}}\leq T_{\hat{N}^{\sigma,*}_1}\right\}}\right.\\
&+\left.\tilde{U}_{T_{\hat{N}^{\sigma,*}_1}}\mathbbm{1}_{\left\{T_{\hat{N}^{\sigma,*}_1}<T,T_{\hat{N}^{\sigma,*}_1}<
T_{N^{\tau}}\right\}}\right]=\mathbbm{E}\left[\tilde{R}(T_{\hat{N}^{\sigma,*}_1},T_{N^{\tau}})\right],
\end{align*}
for any $\tilde{\mathbbm{G}}$-stopping time
$N^{\tau}\in\mathcal{N}_1(\lambda)$. Taking the supremum over
$N^{\tau}\in\mathcal{N}_1(\lambda)$, we obtain
\[{Q_0^{\lambda}\geq \sup_{N^{\tau}\in\mathcal{N}_1(\lambda)}\mathbbm{E}\left[\tilde{R}(T_{\hat{N}^{\sigma,*}_1},T_{N^{\tau}})\right]\geq \inf_{N^{\sigma}\in\mathcal{N}_1(\lambda)}\sup_{N^{\tau}\in\mathcal{N}_1(\lambda)}\mathbbm{E}\left[\tilde{R}(T_{N^{\sigma}},T_{N^{\tau}})\right]=\overline{q}^{\lambda}.}\]
Similarly, we also have $Q_0^{\lambda}\leq \underline{q}^{\lambda}$.
It then follows from $\overline{q}^{\lambda}\geq
\underline{q}^{\lambda}$ that
$Q_0^{\lambda}=\underline{q}^{\lambda}=\overline{q}^{\lambda}.$

Finally, we verify that
$Q_0^{\lambda}=\mathbbm{E}[\tilde{R}(T_{\hat{N}^{\sigma,*}_1},T_{\hat{N}^{\tau,*}_1})],$
so $(\hat{N}^{\sigma,*}_1,\hat{N}^{\tau,*}_1)$ are the optimal
stopping strategy. Indeed, with $N^{\sigma}=\hat{N}^{\sigma,*}_1$
and $N^{\tau}=\hat{N}^{\tau,*}_1$, (\ref{inequality}) becomes an
equality due to the martingale property (i), i.e.
\begin{eqnarray*}
Q_0^{\lambda}&=&\mathbbm{E}\left[\int_0^{T_1\wedge T}\tilde{f}_s\,ds+\tilde{\xi}\mathbbm{1}_{\{T_1>T\}}+\hat{q}_{T_1}^{\lambda}\mathbbm{1}_{\{T_1\leq T\}}\right]\\
&=&\mathbbm{E}\left[\int_0^{T_1\wedge T}\tilde{f}_s\,ds+\tilde{\xi}\mathbbm{1}_{\{T_1> T\}}+\mathbbm{E}\left[\tilde{R}_1(T_{\hat{N}^{\sigma,*}_1},T_{\hat{N}^{\tau,*}_1})|\mathcal{G}_{T_1}\right]\mathbbm{1}_{\{T_1\leq T\}}\right]\\
&=&\mathbbm{E}\left[\tilde{R}(T_{\hat{N}^{\sigma,*}_1},T_{\hat{N}^{\tau,*}_1})\right].
\end{eqnarray*}
We conclude the proof by proving that the optimal stopping times
$\left(\hat{N}^{\sigma,*}_1,\hat{N}^{\tau,*}_1\right)$ are actually
$\left(N^{\sigma,*}_1,N^{\tau,*}_1\right)$ in
(\ref{optimalStop_discounted}). Indeed,
\begin{eqnarray*}
\hat{N}^{\sigma,*}_1&=&\inf\{N\geq 1:\hat{q}_{T_N}^{\lambda}=\tilde{U}_{T_N}\}\wedge M\\
&=&\inf\{N\geq 1:\hat{Q}_{T_N}^{\lambda}=\tilde{U}_{T_N}\}\wedge M\\
&=&\inf\{N\geq 1:Q_{T_N}^{\lambda}\geq \tilde{U}_{T_N}\}\wedge
M=N^{\sigma,*}_1,
\end{eqnarray*}
and, similarly, $\hat{N}^{\tau,*}_1=N^{\tau,*}_1$.

\section{Connection with standard Dynkin games}

We show that, when $\lambda\rightarrow\infty$, the value
$v^{\lambda}$ of the constrained Dynkin game converges to the value
of a standard Dynkin game. The setup is the same as in section 2
except that the control set is replaced with $\mathcal{R}_t$, which
is defined as
\[\mathcal{R}_t=\{\mathbbm{F}\mbox{-stopping time }\tau\mbox{ for }t\leq \tau(\omega) \leq T\}.\]
Define the corresponding upper and lower values of the {standard}
Dynkin game as
\begin{equation}
\label{upperValues_cts}
\overline{v}=\inf_{\sigma\in\mathcal{R}_0}\sup_{\tau\in\mathcal{R}_0}\mathbbm{E}\left[R(\sigma,\tau)\right],
\end{equation}
\begin{equation}
\label{lowerValues_cts}
\underline{v}=\sup_{\tau\in\mathcal{R}_0}\inf_{\sigma\in\mathcal{R}_0}\mathbbm{E}\left[R(\sigma,\tau)\right].
\end{equation}
This game is said to have value $v$ if
$v=\overline{v}=\underline{v}$, and $(\sigma^*,\tau^*)\in
\mathcal{R}_0\times \mathcal{R}_0$ is called a saddle point of the
game if $\mathbbm{E}\left[R(\sigma^*,\tau)\right]\leq
\mathbbm{E}\left[R(\sigma^*,\tau^*)\right]\leq
\mathbbm{E}\left[R(\sigma,\tau^*)\right]$ for every
$(\sigma,\tau)\in \mathcal{R}_0\times \mathcal{R}_0$.

\begin{proposition}
\label{prop_connection_cts} Suppose that Assumption
\ref{assumption_1} holds and, moreover, both $L$ and $U$ are
continuous and satisfy $L_T\leq\xi\leq U_T$. Then, the value $v$ of
the Dynkin game (\ref{upperValues_cts})-(\ref{lowerValues_cts})
exists and, moreover, $\lim_{\lambda \uparrow \infty}v^{\lambda}=v.$
\end{proposition}
\begin{proof}
To solve the Dynkin game
(\ref{upperValues_cts})-(\ref{lowerValues_cts}), we introduce the
following reflected BSDE defined on a random horizon $[0,T]$:
\begin{equation}
\label{final_game_value_cts} V_{t\wedge T}=\xi+\int_{t\wedge
T}^T(f_s-rV_s)ds+\int_{t\wedge T}^T\,dK_s^+-\int_{t\wedge
T}^T\,dK_s^--\int_{t\wedge T}^T Z_s\,dW_s
\end{equation}
for $t\geq 0$, under the constraints (i) $L_t\leq V_t\leq U_t$, for
$0\leq t\leq T$; (ii)
$\int_0^T\left(V_t-L_t\right)\,dK^+_t=\int_0^T\left(U_t-V_t\right)\,dK^-_t=0$.
By a solution to the reflected BSDE (\ref{final_game_value_cts}), we
mean a triplet of $\mathbbm{F}$-progressively measurable processes
$(V,Z,K)$, where $K:=K^+-K^-$ with $K^+$ and $K^-$  being increasing
processes starting from $K^+_0=K^{-}_0=0$.

It follows from Hamadene et al \cite{hamadene1999infinite} that
(\ref{final_game_value_cts}) is well-posed and admits a unique
solution. Using arguments similar to the ones in Cvitanic and
Karatzas \cite{cvitanic1996backward}, it is standard to show that
the value of the Dynkin game
(\ref{upperValues_cts})-(\ref{lowerValues_cts}) exists and is given
by the solution of the reflected BSDE (\ref{final_game_value_cts}),
i.e. $v=\overline{v}=\underline{v}=V_0$.


To prove the second assertion, we note that BSDE
(\ref{final_game_value}) can be regarded as a sequence of penalized
BSDEs for (\ref{final_game_value_cts}), where the local time
processes $K^+$ and $K^-$ are approximated by
\[K_t^{\lambda,+}:=\int_0^t
\lambda\left(L_s-V_s^{\lambda}\right)^+\,ds;\quad
K_t^{\lambda,-}:=\int_0^t
\lambda\left(V_s^{\lambda}-U_s\right)^+\,ds,\] with
$K^{\lambda}:=K^{\lambda,+}-K^{\lambda,-}$. Since  $\lim_{\lambda
\uparrow
\infty}\mathbbm{E}[\sup_{t\in[0,T]}|V_t^{\lambda}-V_t|^2]=0$ (see,
for example, \cite{hamadene1999infinite} and
\cite{cvitanic1996backward}), the second assertion follows
immediately.
\end{proof}

\section{Replication of constrained Dynkin games}
In this section, we discuss about replication of the constrained
Dynkin game. This provides a foundation for the risk-neutral
valuation of convertible bonds introduced in the next section.

We interpret $\mathbb{P}$ as a risk-neutral probability measure, and
let $\bar{N}_t:=\sum_{n\geq 1}\mathbbm{1}_{\{T_n\leq t\}}-\lambda
t$, $t\geq 0$, be the compensated Poisson martingale. Suppose there
exist $(d+2)$ underlying assets, whose pricing processes follow
\begin{align}
\label{dynamics_S}
dS_t^i&=S_t^i(r-q^i)dt+S_t^i\sigma^i dW_t,\ 1\leq i\leq d;\\
dP_t&=P_{t-}rdt+P_{t-}\bar{\sigma}d\bar{N}_t;\\
dB_t&=B_trdt,
\end{align}
where $r>0$ is the risk-free interest rate, $\bar{\sigma}>0$
represents the volatility of $P$, and $q^i$ and
$\sigma^i:=(\sigma^{ij})_{1\leq j\leq d}$ represent, respectively,
the dividend and volatility of $S^i$. Assume that the volatility
matrix $\sigma:=(\sigma^{ij})_{1\leq i,j\leq d}$ is invertible. The
risky assets $(S^i)_{1\leq i\leq d}$ are the underlying assets used
to hedge the Brownian noise of the game. The risky asset $P$ is used
to hedge the jump risk of the Poisson process. In practice, it could
be the cash flow of a credit default swap delivering payoffs at jump
times $(T_n)_{n\geq 1}$ (see, for example, \cite{Bielecki2008CDS}
for the single jump case). Finally, $B$ represents the risk-free
bank account.

From section \ref{bigTheorem} (Lemmas \ref{firstLemma} and
\ref{secondLemma} in particular), we know that the solution
$V^{\lambda}$ of BSDE (\ref{final_game_value}) provides the values
of the constrained Dynkin game
(\ref{upperValues})-(\ref{lowerValues}) starting at different
Poisson arrival times $T_{n-1}$ for $1\leq n\leq M$, and they
satisfy the recursive equation
\begin{align}\label{DPP}
e^{-rT_{n-1}}V^{\lambda}_{T_{n-1}}=&\
\mathbb{E}\left[\int_{T_{n-1}}^{T_n\wedge T}e^{-rs}f_sds+e^{-rT}\xi
\mathbbm{1}_{\{T_n>
T\}}\right.\\
&\left.+\
e^{-rT_{n}}\min\{U_{T_n},\max\{V^{\lambda}_{T_n},L_{T_n}\}\}\mathbbm{1}_{\{T_n\leq
T\}}|\mathcal{G}_{T_{n-1}}\right].\notag
\end{align}
Thus, the discounted payoff of the game starting at $T_{n-1}$ is
\begin{align}\label{payoff}
&\left(\int_{T_{n-1}}^{T}e^{-rs}{f}_s\,ds+e^{-rT}{\xi}\right)\mathbbm{1}_{\{T_n>T\}}\\
+&\left(\int_{T_{n-1}}^{T_n}e^{-rs}{f}_s\,ds+e^{-rT_{n}}
\min\{U_{T_n},\max\{V^{\lambda}_{T_n},L_{T_n}\}\}\right)
\mathbbm{1}_{\{T_n\leq T\}},\notag
\end{align}
with $V^{\lambda}_{T_n}$ being the value of the game starting at
$T_n$. Compared to the original payoff (\ref{def_payoff}), the above
payoff (with $n=1$) only involves the first Poisson arrival time
$T_1$, and the optimality of stopping strategies is encoded in
$V_{T_1}^{\lambda}$. Thus, the replication of the constrained Dynkin
game (\ref{upperValues})-(\ref{lowerValues}) naturally depends on
the replication of the same game but starting at Poisson arrival
time $T_2$, the later of which in turn depends on the replication of
the game starting from $T_3$ and so on and so forth. In particular,
the discounted payoff of the game starting at $T_{M-1}$ is
$\int_{T_{M-1}}^Te^{-rs}f_sds+e^{-rT}\xi,$ since $T_{M-1}\leq T<T_M$
by the definition of the random variable $M$.

For $1\leq n\leq M $, consider the constrained Dynkin game starting
at Poisson arrival time $T_{n-1}$. We aim to construct a replication
portfolio $(\pi_t^{S,n},\pi_t^{P,n},\pi_t^{B,n})$,
$t\in[T_{n-1},T]$, to replicate the discounted payoff
(\ref{payoff}), where $\pi^{S,n}=(\pi^{S^i,n})_{1\leq i\leq d}$
represent the amount of the money invested in $(S^i)_{1\leq i\leq
d}$, and $\pi^{P,n}$ and $\pi^{B,n}$ represent the amount of the
money invested in $P$ and $B$, respectively. Let $X_t^n$ be the
corresponding wealth of each player at time $t$. Then,
$X_t^n=\sum_{i=1}^d\pi_t^{S^i,n}+\pi_t^{P,n}+\pi_t^{B,n}$, and the
self-financing condition implies that
\begin{align}\label{wealthequ}
X_t^n&=X^n_{T_{n-1}}+\int_{T_{n-1}}^t
\left(\sum_{i=1}^d\frac{\pi_s^{S^i,n}}{S_s^i}dS_s^i+
\frac{\pi_s^{P,n}}{P_{s-}}dP_s+\frac{\pi_s^{B,n}}{B_s}dB_s+
\sum_{i=1}^dq^i\pi_s^{S^i,n}ds\right)\\
&=X^n_{T_{n-1}}+\int_{T_{n-1}}^t\left(rX_s^nds+\pi_s^{S,n}\sigma
dW_s+\pi_s^{P,n}\bar{\sigma}d\bar{N}_s\right),\notag
\end{align}
for $t\in[T_{n-1},T]$. The problem is to find a replication
portfolio $(\pi^{S,n},\pi^{P,n},\pi^{B,n})$ such that the discounted
wealth $e^{-r{T}}X_T^n$ replicates the discounted payoff
(\ref{payoff}), and to prove that
$X^{n}_{T_{n-1}}=V^{\lambda}_{T_{n-1}}$, i.e. the constrained Dynkin
game starting from $T_{n-1}$ is replicable and its value is indeed
given by $V^{\lambda}_{T_{n-1}}$.

\begin{theorem}\label{theorem_5.1}
Let $(Y^{\xi,\theta},Z^{\xi,\theta})$ be the unique solution of the
linear BSDE defined on the random horizon $[\theta,T]$ with a
parameter $\theta\in[0,T]$, i.e.
\begin{equation}\label{BSDE0}
Y_{t\wedge
T}^{\xi,\theta}=\left(\int_{\theta}^Te^{r(T-s)}f_sds+\xi\right)-\int_{t\wedge
T}^TrY_s^{\xi,\theta}ds-\int_{t\wedge T}^TZ_s^{\xi,\theta}dW_s,
\end{equation}
for $t\geq \theta$. Then, for the constrained Dynkin game starting
at $T_{M-1}$, its replication wealth and the corresponding
replication portfolio are given by
\begin{align}\label{wealth0}
&X^M_t=Y^{\xi,T_{M-1}}_t;\\
&(\pi_t^{S,M},\pi_t^{P,M},\pi_t^{B,M})=(Z_t^{\xi,T_{M-1}}\sigma^{-1},0,X_t^{M}-\pi_t^{S,M}),\
t\in[T_{M-1},T],\notag
\end{align}
where $\sigma^{-1}$ is the inverse of the volatility matrix
$(\sigma^{ij})_{1\leq i,j\leq d}$. Moreover, the value of the game
is given by $X^M_{T_{M-1}}=V^{\lambda}_{T_{M-1}}$.

In general, let $(Y^{\theta},Z^{\theta})$ be the unique solution of
the linear BSDE defined on $[\theta,T]$ with a parameter
$\theta\in[0,T]$, i.e.
\begin{equation}\label{BSDE1}
Y_{t\wedge T}^{\theta}=
\left(\int_{\theta}^{T}e^{r(T-s)}{f}_s\,ds+{\xi}\right)
-\int_{t\wedge
T}^T\left[rY_s^{\theta}-\lambda(Y_s^{\theta,s}-Y_s^{\theta})\right]ds
-\int_{t\wedge T}^TZ_s^{\theta}dW_s,
\end{equation}
for $t\geq \theta$, and
$(Y(\theta,\bar{\theta}),Z(\theta,\bar{\theta}))$ be the unique
solution of the linear BSDE defined on $[\bar{\theta},T]$ with
parameters $\theta,\bar{\theta}$ satisfying $0\leq
\theta<\bar{\theta}\leq T$, i.e.
\begin{align}\label{BSDE2}
Y_{t\wedge
T}^{\theta,\bar{\theta}}=&\left(\int_{\theta}^{\bar{\theta}}e^{r(T-s)}f_sds+e^{r(T-\bar{\theta})}\min\{U_{\bar{\theta}},\max\{V^{\lambda}_{\bar{\theta}},L_{\bar{\theta}}\}\}\right)\\
&-\int_{t\wedge T}^TrY_s^{\theta,\bar{\theta}}ds-\int_{t\wedge
T}^TZ_s^{\theta,\bar{\theta}}dW_s,\notag
\end{align}
for $t\geq \theta$, where $V^{\lambda}$ is the unique solution to
BSDE (\ref{final_game_value}). Then, for the constrained Dynkin game
starting at $T_{n-1}$ for $1\leq n\leq M-1$, its replication wealth
and the corresponding replication portfolio are given by
\begin{align}\label{wealth1}
X_t^n&=Y_t^{T_{n-1}}\mathbbm{1}_{\{t<T_n\}}+Y_t^{T_{n-1},T_{n}}\mathbbm{1}_{\{t\geq T_n\}};\\
\pi_t^{S,n}&=\left(Z_t^{T_{n-1}}\mathbbm{1}_{\{t\leq T_n\}}+Z_t^{T_{n-1},T_n}\mathbbm{1}_{\{t>T_n\}}\right)\sigma^{-1};\notag\\
\pi_t^{P,n}&=\left(Y_t^{T_{n-1},t}-Y_t^{T_{n-1}}\right)\mathbbm{1}_{\{t\leq T_n\}}{\bar{\sigma}}^{-1};\notag\\
\pi_t^{B,n}&=X_t-\pi_t^{S,n}-\pi_t^{P,n},\ t\in[T_{n-1},T].\notag
\end{align}
Moreover, the value of the game is given by
$X^n_{T_{n-1}}=V^{\lambda}_{T_{n-1}}$.
\end{theorem}
\begin{proof} We first replicate the constrained Dynkin game starting at
$T_{M-1}$. It is clear that
$(Y^{\xi,T_{M-1}},Z^{\xi,T_{M-1}}\sigma^{-1},0)$ satisfies the
wealth equation (\ref{wealthequ}) and, moreover, by applying It\^o's
formula to $e^{-rt}Y_t^{\xi,\theta}$, we further have
\begin{equation*}
e^{-r(t\wedge T)}Y_{t\wedge
T}^{\xi,\theta}=\left(\int_{\theta}^Te^{-rs}f_sds+e^{-rT}\xi\right)
-\int_{t\wedge T}^Te^{-rs}Z_s^{\xi,\theta} dW_s.
\end{equation*}
Thus, $e^{-rT}Y_{T}^{\xi,T_{M-1}}$ replicates the discounted payoff
(\ref{payoff}) with $n=M$. Furthermore,
$$X_{T_{M-1}}^M=Y_{T_{M-1}}^{\xi,T_{M-1}}=\mathbb{E}\left[\left.\int_{T_{M-1}}^Te^{-r(s-T_{M-1})}f_sds+e^{-r(T-T_{M-1})}\xi\right|\mathcal{G}_{T_{M-1}}\right]=V^{\lambda}_{T_{M-1}}.$$

In general, we prove the assertion for $1\leq n\leq M-1$ by
induction. Suppose the assertion holds for the game starting at
$T_{n}$ and $X_{T_n}^{n+1}=V^{\lambda}_{T_n}$. Then, for the game
starting at $T_{n-1}$, by the construction of $X^n$ in
(\ref{wealth1}) and the terminal data for BSDEs (\ref{BSDE1}) and
(\ref{BSDE2}), we have
\begin{align}\label{terminal}
e^{-rT}X_{T}^n=&\
e^{-rT}\left(Y_T^{T_{n-1}}\mathbbm{1}_{\{T<T_n\}}+Y_T^{T_{n-1},T_{n}}\mathbbm{1}_{\{T\geq
T_n\}}\right)\\
=&\left(\int_{T_{n-1}}^{T}e^{-rs}{f}_s\,ds+e^{-rT}{\xi}\right)\mathbbm{1}_{\{T_n>T\}}\notag\\
&+\left(\int_{T_{n-1}}^{T_n}e^{-rs}{f}_s\,ds+e^{-rT_{n}}
\min\{U_{T_n},\max\{V^{\lambda}_{T_n},L_{T_n}\}\}\right)
\mathbbm{1}_{\{T_n\leq T\}}.\notag
\end{align}
Therefore, $e^{-rT}X_T^n$ replicates the discounted payoff
(\ref{payoff}).

Next, we show that $(X^n,\pi^{S,n},\pi^{P,n})$ given in
(\ref{wealth1}) indeed satisfies the wealth equation
(\ref{wealthequ}). To this end, note that
$$X_t^n=X^n_{t\wedge {T_n-}}+(X^n_{t\wedge T_n}-X^n_{t\wedge {T_n-}})+(X_t^n-X^n_{t\wedge T_n}),$$
for $t\in[T_{n-1},T]$. Since $X^n_{t\wedge {T_n-}}=Y_{t\wedge
{T_n-}}^{T_{n-1}}$ by the definition of $X^n$, we have
\begin{align*}
X^n_{t\wedge {T_n-}}=&\ Y_{
{T_{n-1}}}^{T_{n-1}}+\int_{T_{n-1}}^{t\wedge
T_{n}-}\left[rY_s^{T_{n-1}}-\lambda(Y_s^{T_{n-1},s}-Y_s^{T_{n-1}})\right]ds+\int_{
T_{n-1}}^{t\wedge T_{n}-}Z_s^{T_{n-1}}dW_s\\
=&\ X_{{T_{n-1}}}^{n}+\int_{T_{n-1}}^{t\wedge
T_{n}-}rY_s^{T_{n-1}}ds+\int_{ T_{n-1}}^{t\wedge
T_{n}-}Z_s^{T_{n-1}}dW_s\\
&-\int_{T_{n-1}}^{t}(Y_s^{T_{n-1},s}-Y_s^{T_{n-1}})\mathbbm{1}_{\{s\leq
T_n \}}\lambda ds.
\end{align*}
Furthermore, at the Poisson arrival time $T_n$, $X^n$ has a jump
with size
$$X^n_{t\wedge T_n}-X^n_{t\wedge
{T_n-}}=\left(Y^{T_{n-1},T_n}_{T_n}-Y^{T_{n-1}}_{
T_{n}-}\right)\mathbbm{1}_{\{T_n\leq
t\}}=\int_0^t(Y_s^{T_{n-1},s}-Y_s^{T_{n-1}})\mathbbm{1}_{\{s\leq T_n
\}})dN_s.$$ On the other hand, since $X_t^n=Y_t^{T_{n-1},T_n}$ on
the event $\{t\geq T_n\}$, we have
$$X_t^n-X^n_{t\wedge T_n}=\int_{t\wedge T_n}^{t}
rY_s^{T_{n-1},T_n}ds+ \int_{t\wedge T_n}^{t}Z_s^{T_{n-1},T_n}dW_s.$$
In turn, we deduce, using the constructions of $X^n$, $\pi^{S,n}$
and $\pi^{P,n}$ in (\ref{wealth1}), that
\begin{align*}
X_t^n= X_{T_{n-1}}^n+\int_{T_{n-1}}^trX_s^nds+
\int_{T_{n-1}}^t\pi_s^{S,n}\sigma
dW_s+\int_{T_{n-1}}^{t}\pi_s^{P,n}\bar{\sigma}d\bar{N}_s.
\end{align*}

Finally, applying It\^o's formula to $e^{-rt}X_t^{n}$ and using
(\ref{terminal}), we obtain that
\begin{align*}
e^{-r(t\wedge T)}X_{t\wedge T}^{n}=&\left(\int_{T_{n-1}}^{T}e^{-rs}{f}_s\,ds+e^{-rT}{\xi}\right)\mathbbm{1}_{\{T_n>T\}}\notag\\
&+\left(\int_{T_{n-1}}^{T_n}e^{-rs}{f}_s\,ds+e^{-rT_{n}}
\min\{U_{T_n},\max\{V^{\lambda}_{T_n},L_{T_n}\}\}\right)
\mathbbm{1}_{\{T_n\leq T\}}\\
&-\int_{t\wedge T}^Te^{-rs}\pi_s^{S,n}\sigma dW_s-\int_{t\wedge
T}^{T}e^{-rs}\pi_s^{P,n}\bar{\sigma}d\bar{N}_s.
\end{align*}
In turn, taking conditional expectation with respect to
$\mathcal{G}_{T_{n-1}}$ and using (\ref{DPP}), we conclude that
$X_{T_{n-1}}^{n}=V_{T_{n-1}}^{\lambda}$.
\end{proof}

\section{Application to convertible bonds with random intervention times}
\label{pcb_poisson_section}

In this section, using the constrained Dynkin game introduced in
section 2, we study convertible bonds for which both players are
only allowed to stop at a sequence of random intervention times.

Traditionally, convertible bond models often assume that both the
bond holder and the issuing firm are allowed to stopped at any
stopping time adapted to the firm's fundamental (such as its stock
prices). In reality, there may exist some liquidation constraint as
an external shock, and both players only make their decisions when
such a shock arrives. We model such a liquidation shock as the
arrival times of an exogenous Poisson process. A similar idea has
first appeared in the modeling of debt run problems (see
\cite{Liang2}), which can be formulated as optimal stopping problems
with Poisson arrival times.

\begin{assumption}\label{assumption_pcb}
Let $d=1.$ The firm's stock price $S^s$, under the risk-neutral
probability measure $\mathbbm{P}$, follows
\begin{equation}
\label{dynamics_S_1} S_t^s=s+\int_0^t(r-q)S_u^s\,du+\int_0^t\sigma
S_u^s\,dW_u,
\end{equation}
with $S_0^s=s>0$, where the constants $r$, $q$, $\sigma$ represent
the risk-free interest rate, the dividend rate and the volatility of
the stock, satisfying the parameter assumption $r> q$\footnote{The
case $r\leq q$ can be treated in a similar way.}.
\end{assumption}

The firm issues convertible bonds as perpetuities with a constant
coupon rate $c$. Consider an investor purchasing a share of this
convertible bond at initial time $t=0$. By holding the convertible
bond, the investor will continuously receive the coupon rate $c$
from the firm until the contract is terminated. The investor has the
right to convert her bond to the firm's stocks, while the firm has
the right to call the bond and force the bondholder to surrender her
bond to the firm at a sequence of Poisson arrival times
$\{T_n\}_{n\geq 1}$ with a constant intensity $\lambda>0$. Hence,
there are two situations that the contract maybe terminated:

(i) if the firm calls the bond at some $\mathbbm{G}$-stopping time
$\sigma$ firstly, the bondholder will receive a pre-specified
surrender price $K$ at time $\sigma$;

(ii) if the investor chooses to convert her bond at some
$\mathbbm{G}$-stopping time $\tau$ firstly or both players choose to
stop the contract simultaneously, the bondholder will obtain $\gamma
S_{\tau}$ at time $\tau$ from converting her bond with a
pre-specified conversion rate $\gamma\in(0,1)$.

In summary, the investor will obtain the following discounted payoff
at initial time $t=0$:
\begin{equation}
\label{convertible_bond_payoff}
P(s;\sigma,\tau)=\int_0^{\sigma\wedge \tau}
e^{-ru}c\,du+e^{-r\tau}\gamma
S^s_{\tau}\mathbbm{1}_{\{\tau\leq\sigma\}}+e^{-r\sigma}K\mathbbm{1}_{\{\sigma<\tau\}},
\end{equation}
with $\sigma,\tau\in\tilde{\mathcal{R}}_{T_1}(\lambda)$, where
\[\tilde{\mathcal{R}}_{T_i}(\lambda)=\{\mathbbm{G}\mbox{-stopping time }\tau\mbox{ for }\tau(\omega)=T_N(\omega)\mbox{ where }N\geq i\}.\]

The investor will choose $\tau\in\tilde{\mathcal{R}}_{T_1}(\lambda)$
to maximize the bond value, while the firm will choose
$\sigma\in\tilde{\mathcal{R}}_{T_1}(\lambda)$ to maximize the equity
value of the firm by minimizing the bond value. This leads to a
constrained Dynkin game as introduced in section 2. The upper value
and lower value of this \emph{constrained convertible bond} are
\begin{equation}
\label{upperValues_pcb}
\overline{v}^{\lambda}(s)=\inf_{\sigma\in\tilde{\mathcal{R}}_{T_1}(\lambda)}\sup_{\tau\in\tilde{\mathcal{R}}_{T_1}(\lambda)}\mathbbm{E}\left[P(s;\sigma,\tau)\right],
\end{equation}
\begin{equation}
\label{lowerValues_pcb}
\underline{v}^{\lambda}(s)=\sup_{\tau\in\tilde{\mathcal{R}}_{T_1}(\lambda)}\inf_{\sigma\in\tilde{\mathcal{R}}_{T_1}(\lambda)}\mathbbm{E}\left[P(s;\sigma,\tau)\right].
\end{equation}

Note that the constrained Dynkin game in section 2 does not exactly
cover the above constrained convertible bond, since the model in
section 2 has a random terminal time $T$, while the convertible bond
is perpetual. However, in the following proposition, we shall show
that when
$$s\geq
\bar{s}^{\lambda}:=\frac{q+\lambda}{r+\lambda}\frac{K}{\gamma},$$
the optimal stopping strategy is trivial. In this region, it is
always optimal for {both the investor and the firm to stop at the
first Poisson arrival time. Intuitively, when the stock price is
high, the stock is attractive enough to lead both the investor to
convert her bond to stocks and the firm to prevent the investor from
converting by calling the bond as early as possible.}


%

\begin{proposition}
\label{pcb_poisson_alwaysConvertLemma} Suppose {that} Assumption
\ref{assumption_pcb} holds. Then, the value of the {constrained}
convertible bond, denoted as $v^{\lambda}(s)$, exists and satisfies
$L^{\lambda}(s)\leq v^{\lambda}(s)\leq U^{\lambda}$ for
$s\in(0,\infty)$, where
$$L^{\lambda}(s):=\frac{c}{r+\lambda}+ \frac{\lambda}{q+\lambda}\gamma
s;\quad U^{\lambda}:=\frac{c+\lambda K}{r+\lambda}.$$ Moreover, in
the domain $s\in\left[\bar{s}^{\lambda},\infty\right)$, it holds
that $v^{\lambda}(s)=L^{\lambda}(s)$, and the optimal stopping
strategy is {$\tau^{*,\lambda}=\sigma^{*,\lambda}=T_1$}.
\end{proposition}
\begin{proof}
Choosing $\tau\equiv T_1$ in (\ref{lowerValues_pcb}) yields a lower
bound of the convertible bond price:
\begin{eqnarray*}
\underline{v}^{\lambda}(s)&=&\sup_{\tau\in\tilde{\mathcal{R}}_{T_1}(\lambda)}\inf_{\sigma\in\tilde{\mathcal{R}}_{T_1}(\lambda)}\mathbbm{E}\left[\int_0^{\sigma\wedge \tau} e^{-ru}c\,du+e^{-r\tau}\gamma S^s_{\tau}\mathbbm{1}_{\{\tau\leq\sigma\}}+e^{-r\sigma}K\mathbbm{1}_{\{\sigma<\tau\}}\right]\\
&\geq &\inf_{\sigma\in\tilde{\mathcal{R}}_{T_1}(\lambda)}\mathbbm{E}\left[\int_0^{T_1} e^{-ru}c\,du+e^{-rT_1}\gamma S^s_{T_1}\right]\\
&=&\mathbbm{E}\left[\int_0^{\infty}\lambda e^{-\lambda m}\left(\int_0^m e^{-ru}c\,du+e^{-rm}\gamma S_m^s\right)\,dm\right]\\
&=&\int_0^{\infty}\lambda e^{-\lambda m}\int_0^m e^{-ru}c\,du\,dm+ \lambda\gamma \mathbbm{E}\left[\int_0^{\infty}e^{-(r+\lambda)m}S_m^s\,dm\right]\\
&=&\frac{c}{r+\lambda}+ \frac{\lambda}{q+\lambda}\gamma
s=L^{\lambda}(s),
\end{eqnarray*}
where we used the integration by parts in the last equality.

On the other hand, by choosing $\sigma\equiv T_1$ in
(\ref{upperValues_pcb}), we get an upper bound of the convertible
bond price:
\begin{eqnarray*}
\overline{v}^{\lambda}(s)&=&\inf_{\sigma\in\tilde{\mathcal{R}}_{T_1}(\lambda)}\sup_{\tau\in\tilde{\mathcal{R}}_{T_1}(\lambda)}\mathbbm{E}\left[\int_0^{\sigma\wedge \tau} e^{-ru}c\,du+e^{-r\tau}\gamma S^s_{\tau}\mathbbm{1}_{\{\tau\leq\sigma\}}+e^{-r\sigma}K\mathbbm{1}_{\{\sigma<\tau\}}\right]\\
&\leq &\sup_{\tau\in\tilde{\mathcal{R}}_{T_1}(\lambda)}\mathbbm{E}\left[\int_0^{T_1} e^{-ru}c\,du+e^{-r{T_1}}\gamma S^s_{T_1}\mathbbm{1}_{\{\tau={T_1}\}}+e^{-r{T_1}}K\mathbbm{1}_{\{\tau>{T_1}\}}\right]\\
&=&\frac{c}{r+\lambda}+ \max\left\{\frac{\lambda}{q+\lambda}\gamma
s,\frac{\lambda
K}{r+\lambda}\right\}=\max\{L^{\lambda}(s),U^{\lambda}\}.
\end{eqnarray*}

In the domain $s\in\left[\bar{s}^{\lambda},\infty\right)$, we always
have $L^{\lambda}(s)\geq U^{\lambda}$ so
$\overline{v}^{\lambda}(s)\leq L^{\lambda}(s)\leq
\underline{v}^{\lambda}(s).$ Thus, the value of the convertible bond
exists, and
$v^{\lambda}(s)=\bar{v}^{\lambda}(s)=\underline{v}^{\lambda}(s)=L^{\lambda}(s),$
with the optimal stopping strategy
{$\tau^{*,\lambda}=\sigma^{*,\lambda}=T_1$}.

In the domain $s\in\left(0,\bar{s}^{\lambda}\right)$, we have
$L^{\lambda}(s)<U^{\lambda}$. Introduce an $\mathbbm{F}$-stopping
time
\[\theta^{\lambda}:=\inf\{u\geq 0:S_u^s\geq \bar{s}^{\lambda}\}.\]
Then, it follows from the dynamic programming principle that
\begin{eqnarray*}
\overline{v}^{\lambda}(s)&=&\inf_{\sigma\in\tilde{\mathcal{R}}_{T_1}(\lambda)}\sup_{\tau\in\tilde{\mathcal{R}}_{T_1}(\lambda)}\mathbbm{E}\left[\int_0^{\sigma\wedge
\tau\wedge \theta^{\lambda}} e^{-ru}c\,du+
e^{-r\theta^{\lambda}}{v}^{\lambda}\left(S^s_{\theta^{\lambda}}\right)\mathbbm{1}_{\{\sigma
\wedge \tau\geq \theta^{\lambda}\}}
\right.\\
&&+\left(e^{-r\tau}\gamma
S^s_{\tau}\mathbbm{1}_{\{\tau\leq\sigma\}}+e^{-r\sigma}K\mathbbm{1}_{\{\sigma<\tau\}}\right)\mathbbm{1}_{\{\sigma
\wedge \tau<\theta^{\lambda}\}}\left. \right].
\end{eqnarray*}
By the definition of the stopping time $\theta^{\lambda}$,
${v}^{\lambda}\left(S^s_{\theta^{\lambda}}\right)={v}^{\lambda}\left(\bar{s}^{\lambda}\right)=L^{\lambda}(\bar{s}^{\lambda})=U^{\lambda}$.
Thus, in the {domain} $s\in(0,\bar{s}^{\lambda})$,
(\ref{upperValues_pcb})-(\ref{lowerValues_pcb}) are equivalent to
\begin{equation}
\label{upperValues_pcb_revised}
\overline{v}^{\lambda}(s)=\inf_{\sigma\in\tilde{\mathcal{R}}_{T_1}(\lambda)}\sup_{\tau\in\tilde{\mathcal{R}}_{T_1}(\lambda)}\mathbbm{E}\left[\tilde{P}(s;\sigma,\tau)\right],
\end{equation}
\begin{equation}
\label{lowerValues_pcb_revised}
\underline{v}^{\lambda}(s)=\sup_{\tau\in\tilde{\mathcal{R}}_{T_1}(\lambda)}\inf_{\sigma\in\tilde{\mathcal{R}}_{T_1}(\lambda)}\mathbbm{E}\left[\tilde{P}(s;\sigma,\tau)\right],
\end{equation}
where the payoff $\tilde{P}(s;\sigma,\tau)$ is
\begin{equation*}
\label{convertible_bond_payoff_revised} \int_0^{\sigma\wedge
\tau\wedge \theta^{\lambda}}
e^{-ru}c\,du+e^{-r\theta^{\lambda}}U^{\lambda}\mathbbm{1}_{\{\sigma\wedge\tau\geq
\theta^{\lambda}\}}+e^{-r\tau}\gamma
S^s_{\tau}\mathbbm{1}_{\{\tau<\theta^{\lambda},\tau\leq\sigma\}}+e^{-r\sigma}K\mathbbm{1}_{\{\sigma<\theta^{\lambda},\sigma<\tau\}}.
\end{equation*}
Note that if we introduce the $\mathbbm{G}$-stopping time
\begin{equation}\label{T_M}
T_M:=\inf\{T_N\geq \theta^{\lambda}:N\geq 1\},
\end{equation}
since the payoff function $\tilde{P}(s;\sigma,\tau)$ does not change
after $T_M$, we may replace the control set
$\tilde{R}_{T_1}(\lambda)$ in
(\ref{upperValues_pcb_revised})-(\ref{lowerValues_pcb_revised}) with
$\mathcal{R}_{T_1}(\lambda)$, the latter of which consists of
$\mathbb{G}$-stopping times $T_1,T_2,\cdots, T_{M}$.

Now, we apply Theorem \ref{bigTheorem} with $T=\theta^{\lambda}$,
$L_t=\gamma S_t^s$, $U_t=K$, $f_t=c$ and $\xi=U^{\lambda}$ to
(\ref{upperValues_pcb_revised})-(\ref{lowerValues_pcb_revised}), and
obtain the existence of the value of the convertible bond in the
domain $s\in(0,\bar{s}^{\lambda})$.
\end{proof}

Thanks to the above proposition, \emph{we focus our analysis to the
domain $s\in\left(0,\bar{s}^{\lambda}\right)$ in the rest of this
section}. We characterize the value of the convertible bond and the
corresponding optimal stopping strategy via the solution of ODEs and
the associated free boundaries, respectively.



\begin{proposition}
\label{vi_pcb_thm} Suppose {that} Assumption \ref{assumption_pcb}
holds. Define the infinitesimal generator
$\mathcal{L}_0=\frac{1}{2}\sigma^2s^2\partial^2_{ss}+(r-q)s\partial_s-r.$
For $s\in(0,\bar{s}^{\lambda})$, the value of the convertible bond
$v^{\lambda}(s)$ is the unique solution to the following ODEs:

(i) If $c>qK$, then $v^{\lambda}(s)>\gamma s$, and
\begin{equation}
\label{vi_pcb_eq2}
-\mathcal{L}_0v^{\lambda}=c-\lambda(v^{\lambda}-K)^+
\end{equation}
with the boundary condition
$v^{\lambda}(\bar{s}^{\lambda})=U^{\lambda}$;

(ii) If $c<rK $, then $v^{\lambda}(s)<K$, and
\begin{equation}
\label{vi_pcb_eq3} -\mathcal{L}_0v^{\lambda}=c+\lambda(\gamma
s-v^{\lambda})^+
\end{equation}
with the boundary condition
$v^{\lambda}(\bar{s}^{\lambda})=U^{\lambda}$.
\end{proposition}
\begin{proof}
It is immediate from Theorem \ref{bigTheorem} and
(\ref{upperValues_pcb_revised})-(\ref{lowerValues_pcb_revised}) that
the convertible bond value is $v^{\lambda}(s)=V_0^{\lambda,s},$ for
$s\in(0,\bar{s}^{\lambda})$, where $V^{\lambda,s}$ is the first
component of the solution to the penalized BSDE
\begin{equation}
\label{final_game_value_pcb} V_{t\wedge
\theta^{\lambda}}^{\lambda,s}=U^{\lambda}+\int_{t\wedge
\theta^{\lambda}}^{\theta^{\lambda}}\left[c+\lambda\left(\gamma
S^s_u-V_u^{\lambda,s}\right)^+-\lambda\left(V_u^{\lambda,s}-K\right)^+-rV_u^{\lambda,s}\right]du-\int_{t\wedge
\theta^{\lambda}}^{\theta^{\lambda}}Z_u^{\lambda,s}\,dW_u.
\end{equation}
Moreover, the optimal {stopping strategy is}
\begin{equation}
\label{optimalStop_pcb} \left\{\begin{array}{l}
\sigma^{*,\lambda}=\inf\{T_N\geq T_1:V_{T_N}^{\lambda,s}\geq K\}\wedge T_M;\\
\tau^{*,\lambda}=\inf\{T_N\geq T_1:V_{T_N}^{\lambda,s}\leq \gamma
S^s_{T_N}\}\wedge T_M,
\end{array}\right.
\end{equation}
with $T_M$ given in (\ref{T_M}).

On the other hand, by the Markov property of the stock price $S$,
$V^{\lambda,s}_t=v^{\lambda}(S_t^s)$. In turn, It\^o's formula
further implies that
\begin{equation}\label{Ito}
v^{\lambda}(S^s_{\theta^{\lambda}})-v^{\lambda}(S^s_{t\wedge
\theta^{\lambda}})=\int_{t\wedge\theta^{\lambda}}^{\theta^{\lambda}}\left[\mathcal{L}_0v^{\lambda}(S_u^s)
+rv^{\lambda}(S_u^s)\right]du+\int_{t\wedge\theta^{\lambda}}^{\theta^{\lambda}}\sigma
s\partial_{s}v^{\lambda}(S_u^s)dW_u.
\end{equation}
It then follows from (\ref{final_game_value_pcb}) and (\ref{Ito})
that $v^{\lambda}(s)$, for $s\in(0,\bar{s}^{\lambda})$, solves the
ODE
\begin{equation}
\label{vi_pcb_first_eq} -\mathcal{L}_0v^{\lambda}=c+\lambda(\gamma
s-v^{\lambda})^+-\lambda(v^{\lambda}-K)^+,
\end{equation}
with the boundary condition
$v^{\lambda}(\bar{s}^{\lambda})=U^{\lambda}.$ Note that if $c<rK$,
Proposition \ref{pcb_poisson_alwaysConvertLemma} yields
\[v^{\lambda}(s)\leq U^{\lambda}=\frac{c+\lambda K}{r+\lambda}< \frac{rK+\lambda K}{r+\lambda}=K,\]
and if $c>qK$, it follows that
\[v^{\lambda}(s)\geq L^{\lambda}(s)=\frac{c}{r+\lambda}+ \frac{\lambda}{q+\lambda}\gamma s>\frac{qK}{r+\lambda}+ \frac{\lambda}{q+\lambda}\gamma s > \gamma s.\]
The ODEs (\ref{vi_pcb_eq2})-(\ref{vi_pcb_eq3}) then follow
immediately.
\end{proof}

The rest of this section is devoted to the characterization of the
optimal stopping strategy of the constrained convertible bond via
its associated free boundaries.


\subsection{The Case I: $qK< c<rK$}
\label{section_poisson_1}

From Proposition \ref{vi_pcb_thm}, when $qK<c<rK$, we always have
$\gamma s< v^{\lambda}(s)<K.$ Thus, following from
(\ref{optimalStop_pcb}), the optimal stopping strategy is
$$\tau^{*,\lambda}=\sigma^{*,\lambda}=T_{M}.$$

Intuitively, when the coupon rate $c$ satisfies $c<rK$, i.e.
$\frac{c}{r}< K$, the firm shall never spend $K$ to call the bond
back, since it only needs to pay the coupon rate $c$ as a perpetual
bond, whose value is $\frac{c}{r}$. Thus, the firm shall never call
until $T_M$.

When the coupon rate $c$ satisfies $c>qK$, i.e. $c> qK > q
\frac{r+\lambda}{q+\lambda}\gamma s > q\gamma s$, the investor shall
never convert her bond into stocks, since the stock dividends she
will receive by holding $\gamma$ shares of the stock are no more
than what she would otherwise receive from the bond coupons. Thus,
the investor shall never {convert} until $T_M$.

In Figure \ref{fig_simulation1}, the bold horizontal line
{$\bar{s}^{\lambda}$} represents the conversion and calling
boundary. We simulate three Poisson times $T_1=0.3$, $T_2=0.5$,
$T_3=0.8$, and two stock price paths. {The investor (and the firm)
will convert (and call) the bond at $T_1$ for the stock path 1.}
They will continue at $T_1$ and $T_2$, and terminate the contract at
$T_3$ for the stock path 2.

We further calculate the convertible bond value by solving the
corresponding ODE explicitly. Note that in such a situation,
$v^{\lambda}=v^{1,\lambda}$ solves
\begin{equation}
\label{vi_pcb_case1_eq} \left\{\begin{array}{rll}
-\mathcal{L}_0v^{1,\lambda}-c&=&0,\ \text{for}\ 0<s<\bar{s}^{\lambda};\\
\displaystyle v^{1,\lambda}(0+)&=&\displaystyle\frac{c}{r};\\
v^{1,\lambda}(\bar{s}^{\lambda})&=&\displaystyle U^{\lambda}.
\end{array}\right.
\end{equation}
We put the perpetual bond value $\frac{c}{r}$ at the boundary
$v^{1,\lambda}(0+):=\lim_{s\downarrow 0}v^{1,\lambda}(s)$, because
in such a situation, there is no motivation for the firm to call or
for the investor to convert the bond.

The general solution of (\ref{vi_pcb_case1_eq}) has the form
$v^{1,\lambda}(s)=A_+s^{\alpha^+}+A_-s^{\alpha^-}+\frac{c}{r}$, for
$0<s<\bar{s}^{\lambda}$, where
\begin{equation}\label{alpha}
\alpha^{\pm}=\frac{-(r-q-\frac{\sigma^2}{2})\pm\sqrt{(r-q-\frac{\sigma^2}{2})^2+2r\sigma^2}}{\sigma^2}.
\end{equation} Since $\alpha^-<0$, we obtain $A_-=0$ by the boundary condition at
$v^{1,\lambda}(0+)$. Using the other boundary condition, we further
obtain
\begin{equation}
\label{vi_pcb_solution_case1_eq1}
v^{1,\lambda}(s)=A^{1,\lambda}s^{\alpha}+\frac{c}{r},
\end{equation}
where $\alpha=\alpha^+$ and $
A^{1,\lambda}=\frac{\lambda}{r+\lambda}\frac{rK-c}{r}\left(\bar{s}^{\lambda}\right)^{-\alpha}.$

In Figure 2, we further plot the value function $v^{1,\lambda}(s)$,
which always stays between $[L^{\lambda}(s),U^{\lambda}]$ for
$s\in(0,\bar{s}^{\lambda})$. Since $L^{\lambda}>\gamma s$ and
$U^{\lambda}<K$, the value function also stays between $(\gamma s,
K)$, which means it is never optimal for the firm or the investor to
stop in the region $s\in(0,\bar{s}^{\lambda})$.


%

\subsection{The Case II: $c\geq rK$}
\label{section_poisson_2}

It is obvious that $c>qK$ if $c\geq rK$. Thus, from Proposition
\ref{vi_pcb_thm},  we always have $v^{\lambda}(s)>\gamma s,$ and
following from (\ref{optimalStop_pcb}), the optimal conversion
strategy for the investor is
$$\tau^{*,\lambda}=T_M,$$
i.e. it is never optimal for the investor to convert until $T_M$.
Instead, the investor's optimal strategy is to keep the convertible
bond to receive its coupons (up to $T_M$).

On the other hand, following from (\ref{vi_pcb_eq2}),
$v^{\lambda}=v^{2,\lambda}$ solves
\begin{equation}
\label{vi_pcb_case2_eq} \left\{\begin{array}{rll}
-\mathcal{L}_0v^{2,\lambda}-c+\lambda(v^{2,\lambda}-K)^+&=&0,\ \text{for}\ 0< s<\bar{s}^{\lambda};\\
v^{2,\lambda}(0+)&=&\displaystyle U^{\lambda};\\
v^{2,\lambda}(\bar{s}^{\lambda})&=&\displaystyle U^{\lambda}.
\end{array}\right.
\end{equation}
We {put} $U^{\lambda}$ at the boundary
$v^{2,\lambda}(0+):=\lim_{s\downarrow 0}v^{2,\lambda}(s)$. In this
situation, since the coupon rate $c$ is too large, the firm would
prefer to convert as soon as possible to stop paying the bond
coupons. It is clear that $v^{2,\lambda}(s)=U^{\lambda}\geq K. $ In
turn, by (\ref{optimalStop_pcb}), it is optimal for the firm to call
as soon as possible, i.e. at the first Poisson arrival time
$$\sigma^{*,\lambda}=T_1.$$



In Figure 3, the bold horizontal line ${\bar{s}^{\lambda}}$
represents the conversion boundary for the investor. Once again, we
simulate three Poisson times ${T_1=0.25}$, $T_2=0.5$, $T_3=0.8$, and
two stock price paths. For the stock price path 1, the firm will
call the bond at $T_1$ firstly, and for the stock price path 2, both
the firm and the investor will terminate the contract at $T_1$.

Figure 4 further plots the value function $v^{2,\lambda}$, which is
a constant $U^{\lambda}$ for $s\in(0,\bar{s}^{\lambda})$. Since the
value function always stays above $K$, and therefore also above
$\gamma s$, it is never optimal for the investor to convert in the
region $(0,\bar{s}^{\lambda})$.

\subsection{The Case III: $c\leq qK$}
\label{section_poisson_3}

It is obvious that $c<rK$ if $c\leq qK$. Thus, from Proposition
\ref{vi_pcb_thm},  we always have $v^{\lambda}(s)<K,$ and following
from (\ref{optimalStop_pcb}), the optimal calling time for the firm
is
$$\sigma^{*,\lambda}=T_M,$$
i.e. it is never optimal for the firm to call until $T_M$.
Furthermore, following from (\ref{vi_pcb_eq3}),
$v^{\lambda}=v^{3,\lambda}$ solves
\begin{equation}
\label{vi_pcb_case3_eq} \left\{\begin{array}{rll}
-\mathcal{L}_0v^{3,\lambda}-c-\lambda(\gamma
s-v^{3,\lambda})^+&=&0,\ \text{for}\ 0< s<\bar{s}^{\lambda};\\
v^{3,\lambda}(0+)&=&\displaystyle \frac{c}{r};\\
v^{3,\lambda}(\bar{s}^{\lambda})&=&\displaystyle U^{\lambda}.
\end{array}\right.
\end{equation}

Next, we solve (\ref{vi_pcb_case3_eq}) explicitly. Since $c\leq qK$,
the intersection point of the lower bound $L^{\lambda}(s)$ of the
convertible bond value and the investor's payoff function $\gamma s$
is no greater than $\bar{s}^{\lambda}$ (so $\gamma s$ is no less
than $L^{\lambda}(s)$ between this intersection point and
$\bar{s}^{\lambda}$). Thus, it may happen that, in the region
$s\in(0,\bar{s}^{\lambda})$, the investor converts the bond earlier
than $T_M$. Since $v^{3,\lambda}(s)>\gamma s$ when $s\downarrow 0$,
and $v^{3,\lambda}(s)\leq \gamma s$ for $s=\bar{s}^{\lambda}$, we
define
\begin{equation}
\label{def_x*_lambda}
{x^{*,\lambda}=\inf\left\{s\in(0,\bar{s}^{\lambda}]:
v^{3,\lambda}(s) \leq \gamma s\right\}.}
\end{equation}
{By definition it is obvious $v^{3,\lambda}>\gamma s$ for $s\in
(0,x^{*,\lambda})$, and by the continuity of $v^{3,\lambda}(\cdot)$,
$v^{3,\lambda}(x^{*,\lambda})=\gamma x^{*,\lambda}$. Let us at the
moment assume that $v^{3,\lambda}\leq \gamma s$ for $s\in
(x^{*,\lambda},\bar{s}^{\lambda}]$. Later, we will verify this
condition. If this condition holds, (\ref{vi_pcb_case3_eq}) is
equivalent to the following free boundary problem}
\begin{eqnarray}
-\mathcal{L}_0v^{3,\lambda}-c&=&0,\ \text{for}\ 0<s<x^{*,\lambda};\label{vi_pcb_case3_eq2} \\
-\mathcal{L}_0v^{3,\lambda}-c+\lambda(v^{3,\lambda}-\gamma s)&=&0,\ \text{for}\ x^{*,\lambda}<s<\bar{s}^{\lambda};\label{vi_pcb_case3_eq3}\\
v^{3,\lambda}(0+)&=&\frac{c}{r};\label{vi_pcb_case3_eq1}\\
v^{3,\lambda}(\bar{s}^{\lambda})&=&U^{\lambda};\label{vi_pcb_case3_eq7}\\
v^{3,\lambda}(x^{*,\lambda}-)&=&\gamma x^{*,\lambda};\label{vi_pcb_case3_eq4}\\
v^{3,\lambda}(x^{*,\lambda}+)&=&\gamma x^{*,\lambda};\label{vi_pcb_case3_eq5}\\
\left(v^{3,\lambda}\right)'(x^{*,\lambda}-)&=&\left(v^{3,\lambda}\right)'(x^{*,\lambda}+).\label{vi_pcb_case3_eq6}
\end{eqnarray}
We first observe that, with the boundary condition
(\ref{vi_pcb_case3_eq1}),  ODEs
(\ref{vi_pcb_case3_eq2})-(\ref{vi_pcb_case3_eq3}) imply
\begin{equation}
\label{vi_pcb_solution_case3_eq1}
v^{3,\lambda}(s)=\left\{\begin{array}{ll}
A^{3,\lambda}s^{\alpha}+\frac{c}{r},&\mbox{ if }s\in(0,x^{*,\lambda});\\
B_+s^{\beta^+}+B_-s^{\beta^-}+\frac{c}{r+\lambda}+\frac{\lambda}{q+\lambda}\gamma
s,&\mbox{ if }s\in(x^{*,\lambda},\bar{s}^{\lambda}),
\end{array}\right.
\end{equation}
where $\alpha=\alpha^+$ is given in (\ref{alpha}),
\begin{equation}
\label{beta}
\beta^{\pm}=\frac{-(r-q-\frac{\sigma^2}{2})\pm\sqrt{(r-q-\frac{\sigma^2}{2})^2+2(r+\lambda)\sigma^2}}{\sigma^2},
\end{equation}
and four unknowns $(A^{3,\lambda},B_+,B_-,x^{*,\lambda})$ are to be
determined. Using the continuity across $x^{*,\lambda}$, i.e.
(\ref{vi_pcb_case3_eq4})-(\ref{vi_pcb_case3_eq5}), the smooth
pasting across $x^{*,\lambda}$, i.e. (\ref{vi_pcb_case3_eq6}), and
the boundary condition at $s=\bar{s}^{\lambda}$, i.e.
(\ref{vi_pcb_case3_eq7}), we obtain that
$x^{*,\lambda}\in\left(0,\bar{s}^{\lambda}\right]$ is the (unique)
solution to the following algebraic equation
\begin{equation}
\label{complicated_eq_x*}
C_1x^{\beta^+-\beta^-+1}+C_2x^{\beta^+-\beta^-}+C_3x+C_4=0,
\end{equation}
with
\begin{equation}
\label{vi_pcb_solution_case3_eq_para} \left\{\begin{array}{l}
C_1=\left(\alpha-\frac{\lambda}{q+\lambda}-\frac{q}{q+\lambda}\beta^+\right)\gamma;\\
C_2=-\left(\alpha\frac{c}{r}-\frac{c}{r+\lambda}\beta^+\right);\\
C_3=-\left(\alpha-\frac{\lambda}{q+\lambda}-\frac{q}{q+\lambda}\beta^-\right)(\bar{s}^{\lambda})^{\beta^+-\beta^-}\gamma;\\
C_4=\left(\alpha\frac{c}{r}-\frac{c}{r+\lambda}\beta^-\right)(\bar{s}^{\lambda})^{\beta^+-\beta^-},
\end{array}\right.
\end{equation}
and the coefficients are determined by
\begin{equation}
\label{vi_pcb_solution_case3_parameters} \left\{\begin{array}{rll}
A^{3,\lambda}&=&\left(x^{*,\lambda}\right)^{-\alpha}\left(\gamma x^{*,\lambda}-\frac{c}{r}\right);\\
B_+&=&\frac{\frac{q}{q+\lambda}\gamma x^{*,\lambda}-\frac{c}{r+\lambda}}{\left(x^{*,\lambda}\right)^{\beta^+}-\left(\bar{s}^{\lambda}\right)^{\beta^+-\beta^-}\left(x^{*,\lambda}\right)^{\beta^-}};\\
B_-&=&\frac{\frac{q}{q+\lambda}\gamma
x^{*,\lambda}-\frac{c}{r+\lambda}}{\left(x^{*,\lambda}\right)^{\beta^-}-\left(\bar{s}^{\lambda}\right)^{\beta^--\beta^+}\left(x^{*,\lambda}\right)^{\beta^+}}.
\end{array}\right.
\end{equation}

{It remains to verify the condition $v^{3,\lambda}\leq \gamma s$ for
$s\in (x^{*,\lambda},\bar{s}^{\lambda}]$. Indeed, since
$A^{3,\lambda}>0$, $\alpha>1$, $B_+<0$, $\beta^+>1$ and $B_->0$,
$\beta^-<0$, it is clear that $v^{3,\lambda}$ is convex in the
interval $(0,x^{*,\lambda})$ and concave in the interval
$(x^{*,\lambda},\bar{s}^{\lambda}]$. Moreover,
$\left(v^{3,\lambda}\right)'(x^{*,\lambda})<\gamma$. This verifies
the condition.}

The optimal {conversion} time for the investor is therefore given as
$$\tau^{*,\lambda}=\inf\{T_N: S^s_{T_N}\geq x^{*,\lambda}\}\wedge T_M.$$


%

In Figure 5, the top bold horizontal line $\bar{s}^{\lambda}$
represents the calling boundary for the firm, and the bottom bold
horizontal line $x^{*,\lambda}$ represents the conversion boundary
for the investor. Once again, we simulate three Poisson times
$T_1=0.3$, $T_2=0.5$, $T_3=0.8$, and two stock price paths. For the
stock price path 1, {both the investor and the firm will terminate
the contract at $T_1$; and for the stock path 2, the investor will
continue at $T_1$ and convert at $T_2$, while the firm will not call
the bond back at neither $T_1$ nor $T_2$}.

In Figure 6, we further plot the value function $v^{3,\lambda}$,
which crosses the payoff function $\gamma s$ in the region
$(0,\bar{s}^{\lambda}]$, so the crossing point $x^{*,\lambda}$ is
the optimal conversion boundary for the investor. Furthermore, the
value function $v^{3,\lambda}$ is strictly dominated by $K$ for
$s\in(0,\bar{s}^{\lambda})$, so the firm never calls the bond back
in this region.

\section{Asymptotics as $\lambda\to\infty$}

We study the asymptotic behavior of the convertible bond price and
its associated free boundaries when the Poisson intensity
$\lambda\rightarrow \infty$. Intuitively, they will converge to
their continuous time counterparts. We prove this intuition in this
section.


\subsection{Review of standard convertible bonds}
The setting is the same as in section \ref{pcb_poisson_section}
except that both the investor and the firm choose their respective
optimal stopping strategies as $\mathbbm{F}$-stopping times taking
values in $[0,\infty]$. Then, the upper and lower values of the
standard convertible bond are given by
\begin{equation}
\label{upperValues_pcb_cts}
\overline{v}=\inf_{\sigma\in\tilde{\mathcal{R}}_0}\sup_{\tau\in\tilde{\mathcal{R}}_0}\mathbbm{E}\left[P(s;\sigma,\tau)\right],
\end{equation}
\begin{equation}
\label{lowerValues_pcb_cts}
\underline{v}=\sup_{\tau\in\tilde{\mathcal{R}}_0}\inf_{\sigma\in\tilde{\mathcal{R}}_0}\mathbbm{E}\left[P(s;\sigma,\tau)\right],
\end{equation}
and the control set $\tilde{\mathcal{R}}_0$ is defined as
\[\tilde{\mathcal{R}}_0=\{\mathbbm{F}\mbox{-stopping time }\tau\mbox{ for }\tau\geq 0\}.\]

We say this game has value $v$ if $v=\overline{v}=\underline{v}$,
and has a saddle point $(\sigma^*,\tau^*)\in
\tilde{\mathcal{R}}_0\times \tilde{\mathcal{R}}_0$ if
$\mathbbm{E}\left[P(s;\sigma^*,\tau)\right]\leq
\mathbbm{E}\left[P(s;\sigma^*,\tau^*)\right]\leq
\mathbbm{E}\left[P(s;\sigma,\tau^*)\right]$ for every
$(\sigma,\tau)\in \tilde{\mathcal{R}}_0\times
\tilde{\mathcal{R}}_0$.

The proof of the following result follows along the similar
arguments in \cite{yan2015dynkin} and is thus omitted. We refer to
\cite{yan2015dynkin} for its further details.

\begin{proposition}
\label{thm_pcb_cts} Suppose that Assumption \ref{assumption_pcb}
holds. Let $\bar{s}:=\frac{K}{\gamma}$, and define an
$\mathbbm{F}$-stopping time $\theta=\inf\{u\geq 0:S_u^s\geq
\bar{s}\}.$ Then, the value of the standard convertible bond $v(s)$
is given as follows:

(i) The Case I: $qK<c<rK$,
\begin{equation}
\label{vi1_pcb_solution} v^{1}(s)=\left\{\begin{array}{ll}
A^{1}s^{\alpha}+\frac{c}{r},&\mbox{ if }s\in(0,\bar{s});\\
\gamma s,&\mbox{ if }s\in[\bar{s},\infty),
\end{array}\right.
\end{equation}
with $\alpha=\alpha^+$ as in (\ref{alpha}) and
$A^1=\frac{rK-c}{r}(\bar{s})^{-\alpha}.$ The optimal stopping
strategy is given by
\begin{equation}
\label{vi1_pcb_solution_optimal} \sigma^{*}=\tau^{*}=\theta.
\end{equation}

(ii) The Case II: $c\geq rK$,
\begin{equation}
\label{vi2_pcb_solution} v^{2}(s)=\left\{\begin{array}{ll}
K,&\mbox{ if }s\in(0,\bar{s});\\
\gamma s,&\mbox{ if }s\in[\bar{s},\infty).
\end{array}\right.
\end{equation}
The optimal stopping strategy is given by
\begin{equation}
\label{vi2_pcb_solution_optimal} \sigma^{*}=0;\quad \tau^{*}=\theta.
\end{equation}

(iii) The Case III: $c\leq qK$,
\begin{equation}
\label{vi3_pcb_solution} v^{3}(s)=\left\{\begin{array}{ll}
A^{3}s^{\alpha}+\frac{c}{r},&\mbox{ if }s\in(0,x^{3});\\
\gamma s,&\mbox{ if }s\in[x^{3},\infty),
\end{array}\right.
\end{equation}
with  $\alpha=\alpha^+$ and $A^3=\left(\gamma
x^3-\frac{c}{r}\right)(x^3)^{-\alpha}.$ The optimal stopping
strategy is given by
\begin{equation}
\label{vi3_pcb_solution_optimal} \sigma^{*}=\theta,\quad
\tau^{*}=\inf\{t\geq 0:S^s_t\geq x^3\},
\end{equation}
where
\[x^3=\left\{\begin{array}{ll}
x^*:=\frac{\alpha}{\alpha-1}\frac{c}{\gamma r},&\mbox{if }c\leq \frac{\alpha-1}{\alpha}rK;\\
\bar{s},&\mbox{if }c>\frac{\alpha-1}{\alpha}rK.
\end{array}\right.\]
\end{proposition}

\subsection{Asymptotics}
We conclude the paper by studying, when $\lambda\rightarrow\infty$,
(i) the convergence of the constrained convertible bond price
$v^{\lambda}$ to its continuous-time counterpart $v$; (ii) the
convergence of the optimal conversion/calling boundaries for the
constrained convertible bond to its continuous-time counterparts.

It is easy to check that $\bar{s}^{\lambda}\rightarrow\bar{s}$,
$A^{1,\lambda}\rightarrow A^1$, $L^{\lambda}(s)\rightarrow \gamma s$
and $U^{\lambda}\rightarrow K$ with the convergence rate $1/\lambda$
by using their explicit forms. As a consequence, we have
$$v^{1,\lambda}(s)\rightarrow v^1(s);\quad
v^{2,\lambda}(s)\rightarrow v^2(s),$$ with the convergence rate
$1/\lambda$. Hence, we only need to establish the convergence
results for Case III when $c\leq qK$.
To this end, we first establish the monotonic property of
$x^{*,\lambda}$, as defined in (\ref{def_x*_lambda}), with respect
to $\lambda$ in the following lemma.

\begin{proposition}
\label{x*_convergence_lemma} Suppose that Assumption
\ref{assumption_pcb} holds and that $c\leq qK$. Then,
$x^{*,\lambda}$  is non-decreasing with respect to $\lambda$.
\end{proposition}
\begin{proof}
By the definition of $x^{*,\lambda}$ in (\ref{def_x*_lambda}), it is
sufficient to prove $v^{3,\lambda}$ is non-decreasing in $\lambda$.
Recall that $v^{3,\lambda}$ is the solution to the ODE
(\ref{vi_pcb_case3_eq}) in the domain $s\in(0,\bar{s}^{\lambda})$,
and $v^{3,\lambda}=L^{\lambda}$ in the domain
$s\in[\bar{s}^{\lambda},\infty)$.

Let us suppose $\lambda_1< \lambda_2$ and it is easy to check that
$\bar{s}^{\lambda_1}< \bar{s}^{\lambda_2}$. For $s\geq
\bar{s}^{\lambda_1}$, we have $v^{3,\lambda_1}=L^{\lambda_1}$. Then,
\begin{eqnarray*}
v^{3,\lambda_1}(s)-v^{3,\lambda_2}(s)&\leq & L^{\lambda_1}(s)-L^{\lambda_2}(s)\\
&=&\frac{c(\lambda_2-\lambda_1)}{(r+\lambda_1)(r+\lambda_2)}-\frac{q(\lambda_2-\lambda_1)}{(q+\lambda_1)(q+\lambda_2)}\gamma s\\
&\leq
&\frac{(q-r)qK(\lambda_2-\lambda_1)}{(r+\lambda_1)(q+\lambda_2)(r+\lambda_2)}<0.
\end{eqnarray*}

On the other hand, for $s<\bar{s}^{\lambda_1}$, note that
$v^{3,\lambda_1}(0+)=v^{3,\lambda_2}(0+)=\frac{c}{r}$ and
$v^{3,\lambda_1}(\bar{s}^{\lambda_1})<v^{3,\lambda_2}(\bar{s}^{\lambda_1})$.
Define the set
$\mathcal{N}=\left\{s\in\left(0,\bar{s}^{\lambda_1}\right):v^{3,\lambda_1}(s)>v^{3,\lambda_2}(s)\right\},$
and suppose that $\mathcal{N}\not=\emptyset$. Then on $\mathcal{N}$,
we have
\[\left\{\begin{array}{l}
-\mathcal{L}_0v^{3,\lambda_1}=c+\lambda_1(\gamma s-v^{3,\lambda_1})^+;\\
-\mathcal{L}_0v^{3,\lambda_2}=c+\lambda_2(\gamma
s-v^{3,\lambda_2})^+,
\end{array}\right.\]
which implies
\begin{eqnarray*}
-\mathcal{L}_0(v^{3,\lambda_1}-v^{3,\lambda_2})&=&\lambda_1(\gamma s-v^{3,\lambda_1})^+-\lambda_2(\gamma s-v^{3,\lambda_2})^+\\
&\leq &\lambda_2\left[(\gamma s-v^{3,\lambda_1})^+-(\gamma
s-v^{3,\lambda_2})^+\right]\leq 0.
\end{eqnarray*}
Hence, we have $v^{3,\lambda_1}\leq v^{3,\lambda_2}$ on
$\mathcal{N}$, which is in contradiction with the definition of
$\mathcal{N}$.
\end{proof}

Since $x^{*,\lambda}$ is bounded by $\bar{s}^{\lambda}(\leq
\bar{s})$, Proposition \ref{x*_convergence_lemma} then implies that
$\lim_{\lambda\to\infty}x^{*,\lambda}$ exists, denoted by
$x^{\infty}$. Moreover, by Proposition \ref{thm_pcb_cts}, we have
$x^{\infty}\leq x^*$ if $c\leq \frac{\alpha-1}{\alpha}rK$, and
$x^{\infty}\leq \bar{s}$ if $c>\frac{\alpha-1}{\alpha}rK$.

On the other hand, by (\ref{complicated_eq_x*}), $x^{*,\lambda}$ is
the solution to the following allergic equation
\begin{align}
\label{complicated_eq_x*_rearrange}
&\left[\left(\frac{x}{\bar{s}^{\lambda}}\right)^{\beta^+-\beta^-}-1\right]\left[\left(\alpha-\frac{\lambda}{q+\lambda}\right)\gamma
x-\alpha\frac{c}{r}-\beta^+\left(\frac{q}{q+\lambda}\gamma
x-\frac{c}{r+\lambda}\right)\right]\\
=&\ (\beta^+-\beta^-)\left(\frac{q}{q+\lambda}\gamma
x-\frac{c}{r+\lambda}\right)\notag.
\end{align}
Sending $\lambda\to\infty$ in (\ref{complicated_eq_x*_rearrange}),
since the right hand side of (\ref{complicated_eq_x*_rearrange}) has
the limit $0$, we obtain
\begin{equation*}
\lim_{\lambda\to\infty}\underbrace{\left[\left(\frac{x^{*,\lambda}}{\bar{s}^{\lambda}}\right)^{\beta^+-\beta^-}-1\right]}_{I^{\lambda}}\underbrace{\left[\left(\alpha-\frac{\lambda}{q+\lambda}\right)\gamma
x^{*,\lambda}-\alpha\frac{c}{r}-\beta^+\left(\frac{q}{q+\lambda}\gamma
x^{*,\lambda}-\frac{c}{r+\lambda}\right)\right]}_{II^{\lambda}}=0.
\end{equation*}
This implies at least one of $I^{\lambda}$ and $II^{\lambda}$ has
the limit 0.

If $c<\frac{\alpha-1}{\alpha}rK$, we have
$\lim_{\lambda\to\infty}I^{\lambda}=-1$, since
\[\lim_{\lambda\to\infty}\frac{x^{*,\lambda}}{\bar{s}^{\lambda}}=\frac{x^{\infty}}{\bar{s}}\leq \frac{x^*}{\bar{s}}=\frac{\alpha}{\alpha-1}\frac{c}{rK}<1.\]
This implies $\lim_{\lambda\to\infty}II^{\lambda}=0$, i.e.
$x^{\infty}=x^*$.

If $c>\frac{\alpha-1}{\alpha}rK$, we have
\[\lim_{\lambda\to\infty}II^{\lambda}=(\alpha-1)\gamma x^{\infty}-\alpha\frac{c}{r}<(\alpha-1)(\gamma x^{\infty}-K)\leq 0,\]
which implies $\lim_{\lambda\to\infty}I^{\lambda}=0$, i.e.
$x^{\infty}=\bar{s}$.

If $c=\frac{\alpha-1}{\alpha}rK$, it is easy to check that
$x^{\infty}=x^*=\bar{s}$.

Hence, we have established the convergence of
$x^{*,\lambda}\rightarrow x^3$ as $\lambda\rightarrow\infty$. As a
consequence, it also follows that $v^{3,\lambda}(s)\rightarrow
v^3(s)$. However, due to the lack of explicit solutions for Case III, it is unclear what is the corresponding convergence rate.\\

\textbf{Acknowledgments}. The author would like to thank the editor,
associate editor, and two referees for their valuable comments and
suggestions on the manuscript.

\newpage


\begin{figure}[h]
  \includegraphics[width=\textwidth]{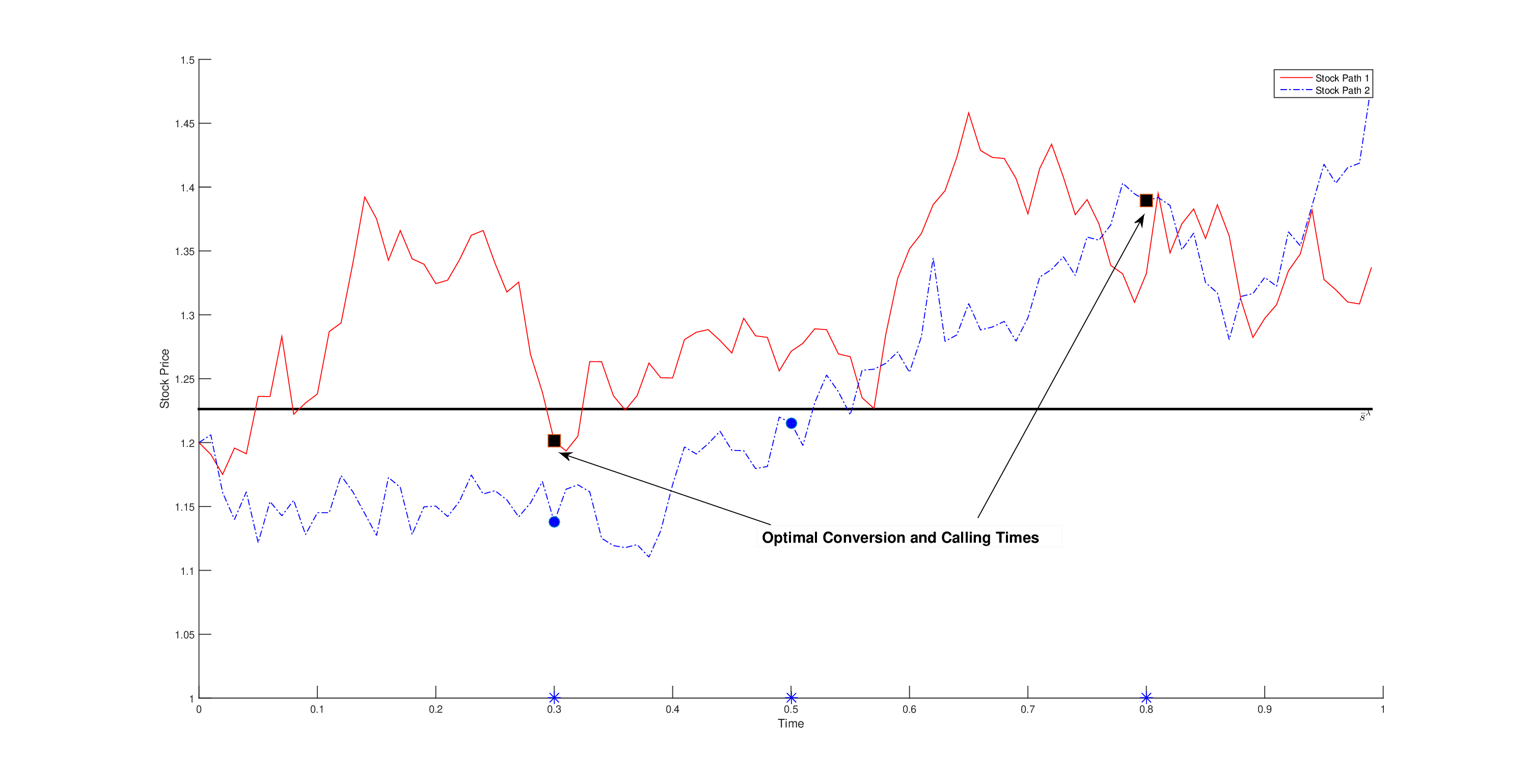}
  \centering
  \caption{Scenario Simulation for Case I. The figure shows two simulated
  stock price paths in the case of $qK<c<rK$.
  Parameter values are $K=1$, $r$=0.05, $q$=0.03, $\sigma$=0.2, $\gamma$=1 and $\lambda$=1.
  The initial stock price is set to $s$=1.2.
  The bold horizontal line describes the conversion and calling boundary $\bar{s}^{\lambda}$.
  Given the Poisson times $T_1$=0.3, $T_2$=0.5 and $T_3$=0.8, the investor will convert and the firm will
  call the bond both at $T_1$ (for path 1) and $T_3$ (for path 2).}
  \label{fig_simulation1}
\end{figure}
\begin{figure}[h]
  \includegraphics[width=\textwidth]{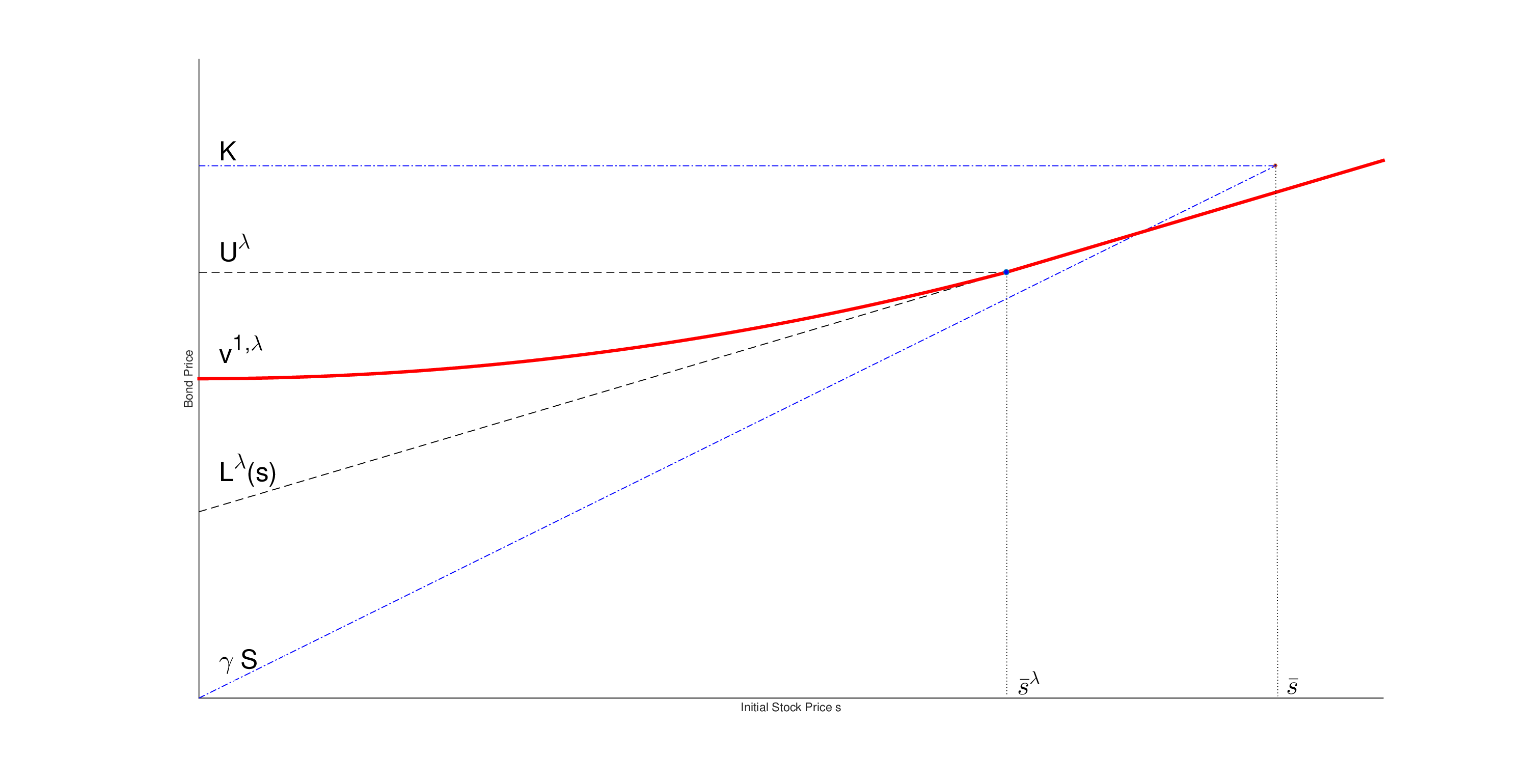}
  \centering
  \caption{The value function $v^{1,\lambda}$ for Case I.}
  \label{fig_sketch1}
\end{figure}
\begin{figure}[h]
  \includegraphics[width=\textwidth]{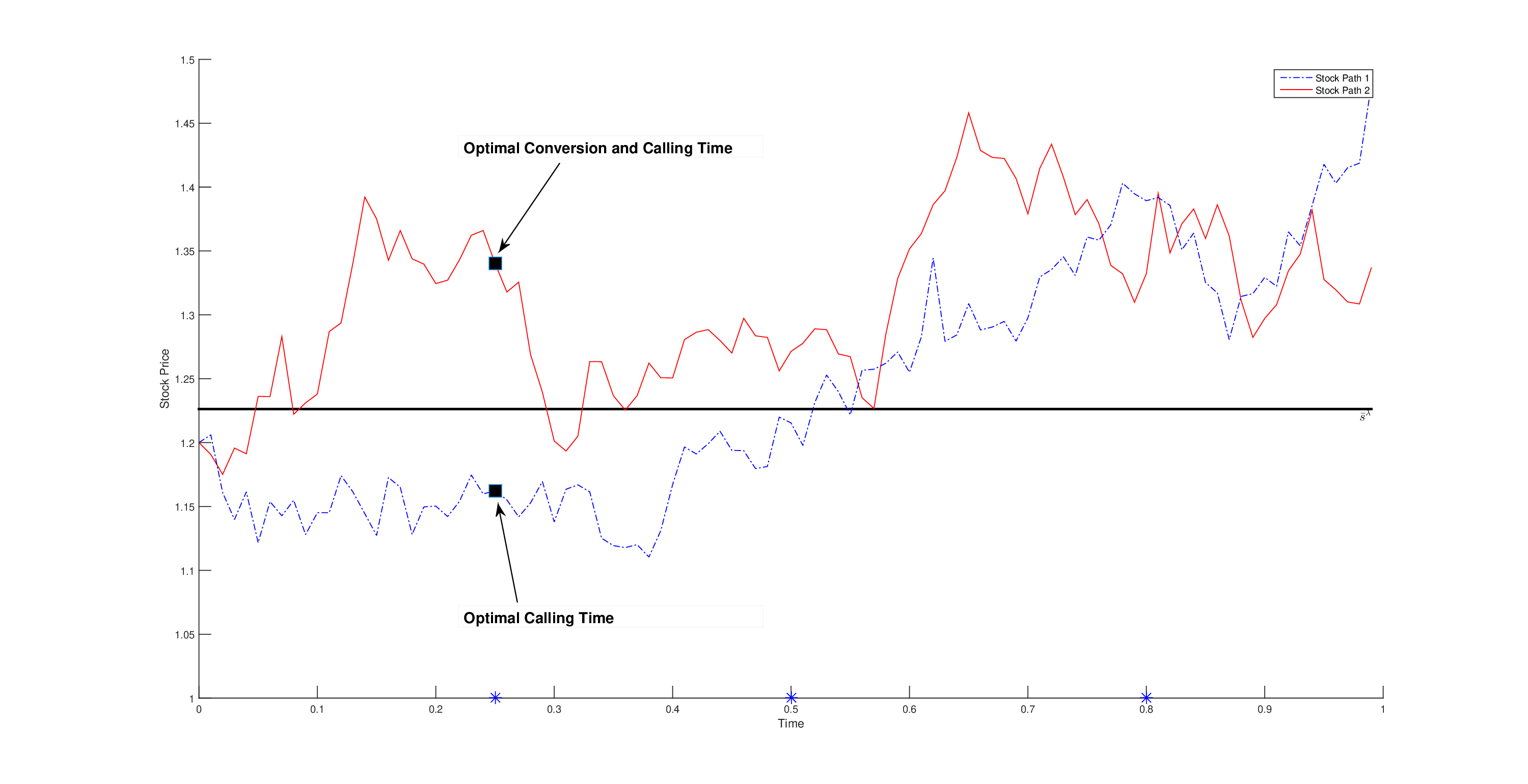}
  \centering
  \caption{Scenario Simulation for Case II. The figure shows two simulated stock price paths in the case of $c\geq rK$. The parameters are the same as those in Figure \ref{fig_simulation1}.The bold horizontal line describes the conversion boundary $\bar{s}^{\lambda}$. Given the Poisson times $T_1$=0.25, $T_2$=0.5 and $T_3$=0.8, the firm will call the bond at $T_1$ (marked square) for the stock price path 1; and both the firm and the investor will terminate the contract at $T_1$ (marked square) for the stock price path 2.}
  \label{fig_simulation2}
\end{figure}
\begin{figure}[h]
  \includegraphics[width=\textwidth]{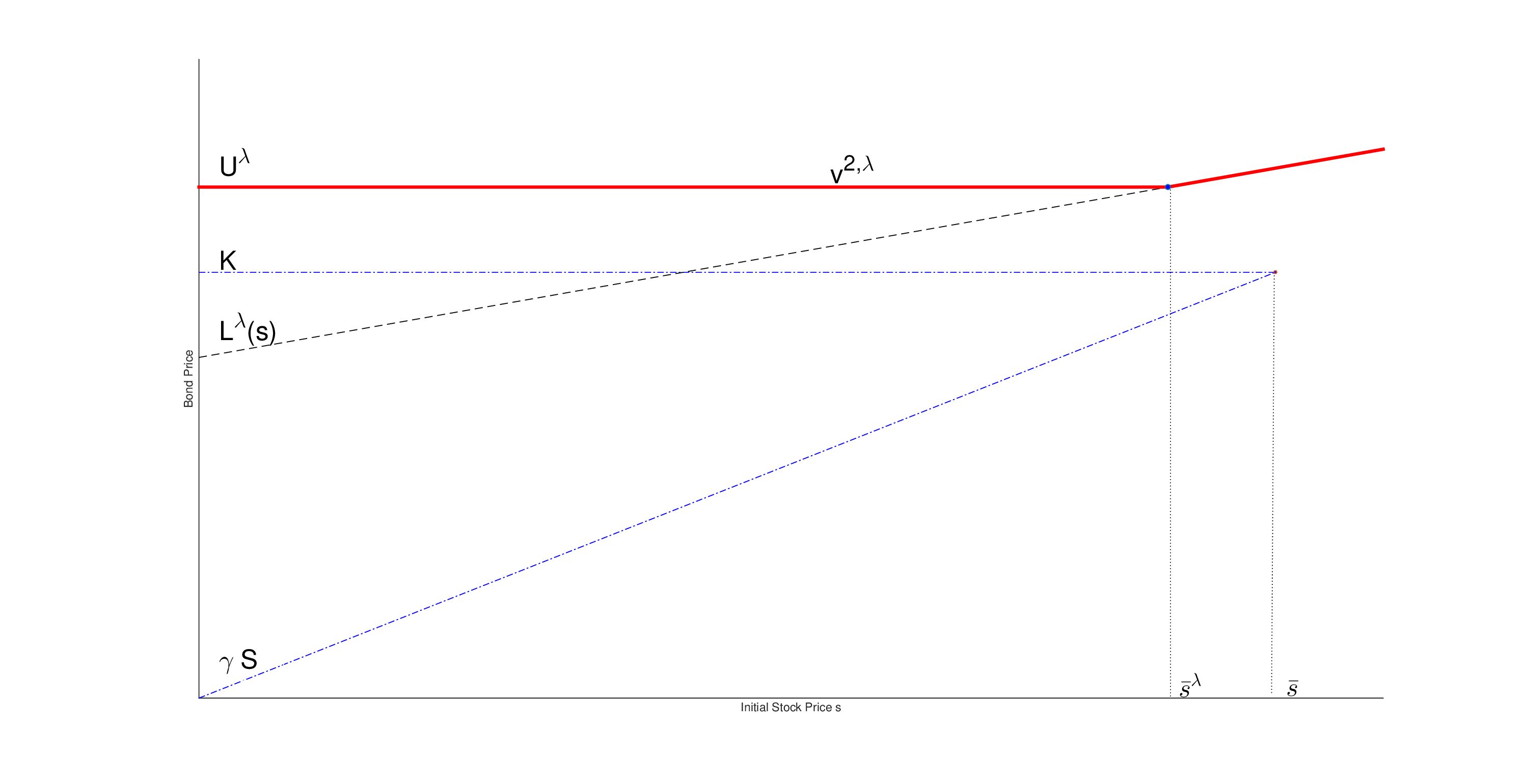}
  \centering
  \caption{The value function $v^{2,\lambda}$ for Case II.}
  \label{fig_sketch2}
\end{figure}
\begin{figure}[h]
  \includegraphics[width=\textwidth]{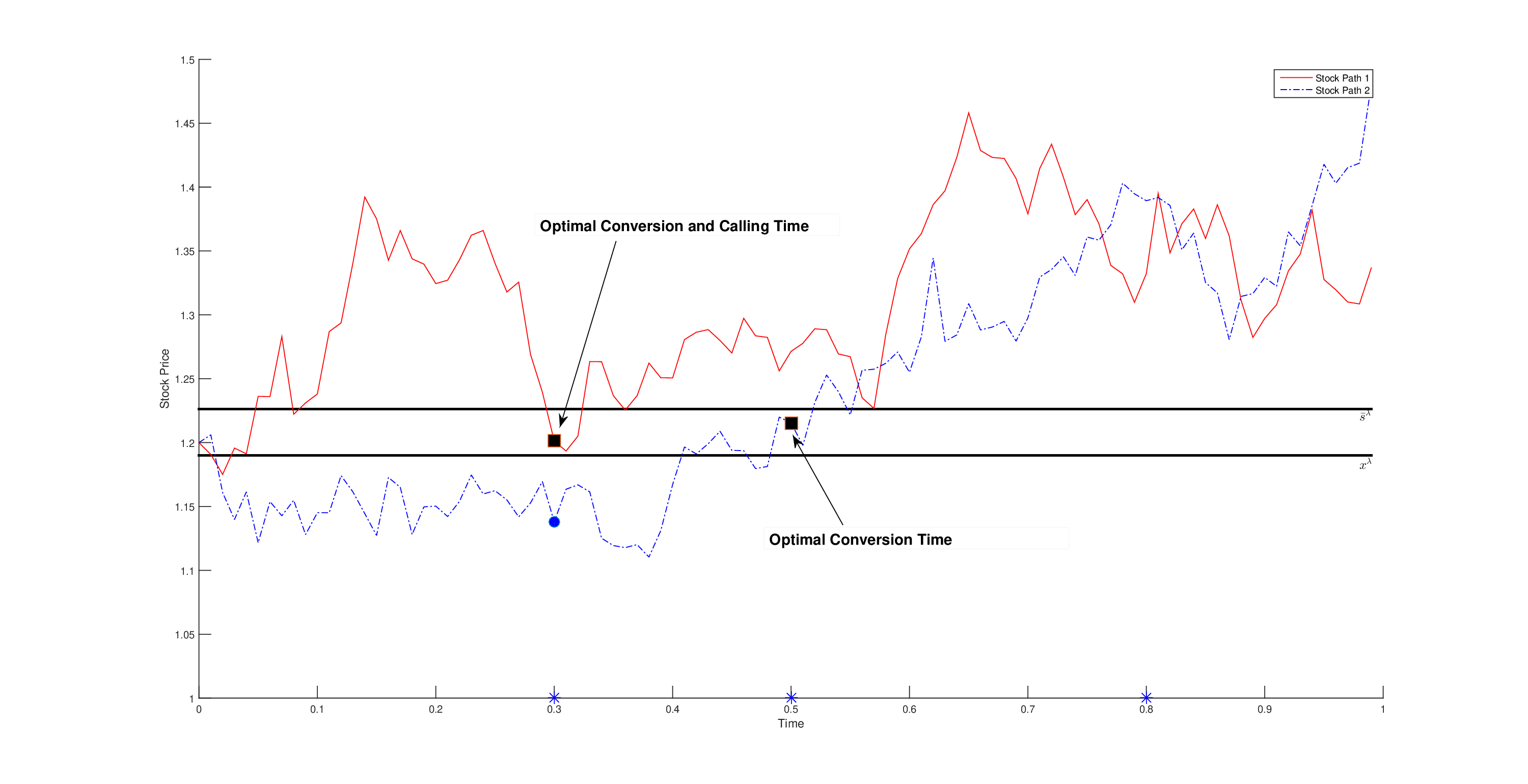}
  \centering
  \caption{Scenario Simulation for Case III. The figure shows two simulated stock price paths in the case of
  $c\leq qK$. The parameters are the same as those in Figure \ref{fig_simulation1}. The top bold horizontal line is the calling boundary $\bar{s}^{\lambda}$, and the bottom bold horizontal line is the conversion boundary $x^{*,\lambda}$. Given the Poisson times $T_1$=0.3, $T_2$=0.5 and $T_3$=0.8, both the investor and the firm will terminate the contract at $T_1$ (marked square) for the stock price path 1; and the invertor will convert her bond $T_2$ (marked square) for the stock price path 2.}
  \label{fig_simulation3}
\end{figure}
\begin{figure}[h]
  \includegraphics[width=\textwidth]{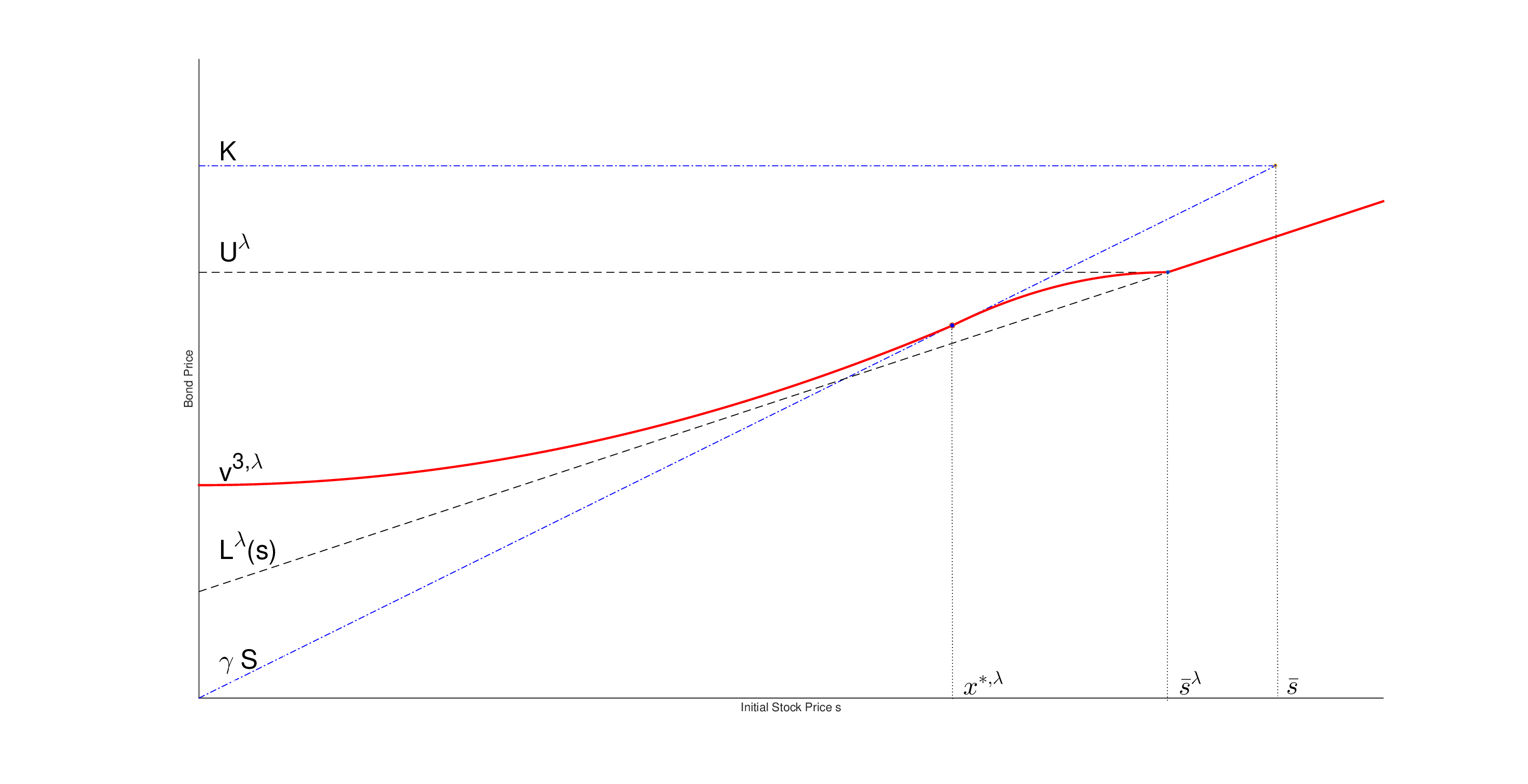}
  \centering
  \caption{The value function $v^{3,\lambda}$ for Case III.}
  \label{fig_sketch3}
\end{figure}

\end{document}